\documentclass[final,1p,times,number]{elsarticle}

\usepackage{amsfonts,amsmath,amscd,amsthm,color}
\usepackage[all]{xy}

\usepackage[utf8]{inputenc}

\usepackage{amsmath}
\usepackage{amssymb}
\usepackage{amsfonts}
\usepackage{amsthm}

\usepackage{ graphicx}
\usepackage{epsfig}

\usepackage{fancyhdr}
\usepackage{anysize}

\usepackage{mathdots}
\usepackage[linesnumbered,ruled]{algorithm2e}
\usepackage{mathdots}

\usepackage[T1]{fontenc}

\usepackage[notref]{showkeys}

\usepackage[english]{babel}

\newtheorem{theorem}{Theorem}[section]
\newtheorem{proposition}[theorem]{Proposition}
\newtheorem{corollary}[theorem]{Corollary}
\newtheorem{definition}[theorem]{Definition}
\newtheorem{lema}[theorem]{Lemma}
\newtheorem{ex}[theorem]{Example}
\newtheorem{exs}[theorem]{Examples}
\newtheorem{obs}[theorem]{Remark}

\numberwithin{equation}{section}

\begin{document}

\begin{frontmatter}

\title{Semialgebraic decomposition \\of real binary forms of a given degree's space}

\

 \author[M]{M. Ansola}\corref{cor2}
 \ead{mansola@ucm.es}

 \author[A]{A. D\'{\i}az-Cano\corref{cor1}}
 \ead{adiazcan@ucm.es}

 \author[MA]{M. A. Zurro\corref{cor3}}
 \ead{mangeles.zurro@uam.es}

 \cortext[cor1]{Corresponding author}
 \address[M]{Universidad Complutense de Madrid}
 \address[A]{Universidad Complutense de Madrid. \\Facultad de Matem\'{a}ticas. IMI and Dpto. de \'{A}lgebra}
 \address[MA]{Universidad Aut\'{o}noma de Madrid}

\begin{abstract}
The Waring Problem over polynomial rings asks for how to decompose a homogeneous polynomial of degree $d$ as a finite sum of $d^{th}$ powers of linear forms.

First, we give a constructive method to obtain a real Waring decomposition of any given real binary form with length at most its degree. Secondly, we adapt the Sylvester's Algorithm to the real case in order to determine a Waring decomposition with minimal length and then we establish its real rank. We use bezoutian matrices to achieve a minimal decomposition.

We consider all real binary forms of a given degree and we decompose this space as a finite union of semialgebraic sets according to their real rank. We study geometrically how distinct Waring decompositions of a fixed form are related. Some explicit  examples are included.
\end{abstract}


\date{\today}

\end{frontmatter}

\tableofcontents

\section{Introduction}

In the $18^{th}$ century, E. Waring proposed as a conjecture (proved by Hilbert in 1909) that every positive integer is the sum of $n$ $k^{th}$ powers of positive integers, with $n$ depending on $k$. For example, four squares, nine cubic powers or nineteen fourth powers. This classical Waring Problem can be extended to polynomial decompositions in this way: any homogeneous polynomial $p$ of degree $d$ in $n$ variables over a field $K$ can be written as the sum of $r$ $d^{th}$ powers of linear forms. When we take $r$ minimal with this property, we call $r$ \textit{the Waring rank} of $p$ over $K$. This expression (not necessa\-rily unique) is known as a Waring decomposition of that polynomial, and it has many applications as much in Applied Mathematics as in Engineering (see \cite{BCM} and the references therein). Applications to Theoretical  Physics can be shown in \cite{BC}. Nowadays this problem is studied as the problem of decomposition of symmetric tensors. Among open problems we find the description in terms of the Waring rank of the space of tensors of given degree and dimension.
\smallskip

Some papers present the study of particular cases, like monomials (for instance, \cite{BCG}, \cite{CKOV} or \cite{O}), but most authors work usually with ``typical forms", i.e., forms whose Waring rank is stable under perturbations of their coefficients. In fact, a rank $r$ is typical for a given degree $d$ if there exists an Euclidean open set in the space of real degree $d$ forms such that any $p$ in such open set has rank $r$. G.$\,$Blekherman \cite{Ble} or P. Common and G. Ottaviani \cite{CO} have  analyzed the ``typical ranks" of general real binary forms. 
\smallskip

The relation between the number of real linear factors and the real Waring rank of binary forms has been also studied by several authors (see \cite{Tok} and the references therein). N. Tokcan in \cite{Tok} has studied the real Waring rank for binary forms from the point of view of their factorization. 
\smallskip

We study a particular  case, that is $K=\mathbb{R}$ and $n=2$. As A. Causa y R. Re affirm in \cite{CR}, the real case becomes more complicated than the complex case. Also \cite{BB} emphasize the importance of the real case for the applications. This real binary case has been recently investigated by different authors (for instance, \cite{BCG}, \cite{CO} or \cite{Re}). It is also known that the complex Waring rank is less than or equal to the real Waring rank (see \cite{BB}, 
where a detailed study of this fact is given).
\smallskip

In this paper we collect in Section 2 the principal definitions and notation we use hereinafter. We include Sylvester's and Borchardt-Jacobi's Theorems. In Section 3 we expound on theoretical concepts that justify our Algorithm, inspired by the Sylvester's one, for Real Waring decomposition (Algorithm \ref{alg_WD}), with little differences in odd or even cases for the rank. Using this Algorithm we can obtain different real Waring decompositions of length less than or equal to $d$ choosing $\frac{d-1}{2}$, if $d$ is odd, or $\frac{d}{2}+1$ if $d$ is even, different parameters that satisfy certain requirements. Several  examples of this Algorithm are shown at the end of the section. 
\smallskip

Section 4 is dedicated to study the Real Waring rank. We present our Real Rank Length's Decomposition Algorithm (see Algorithm \ref{alg_rank}), that guarantees a real Waring decomposition with minimal length and then we can use it to determine the Waring rank of a real binary form. We also exhibit a step-by-step example where differences among complex and real ranks can be observed. Thus, we show how this Algorithm improves the previous one as far as Waring decomposition's length.
\smallskip

In Section 5 we develop the goal of this paper, i. e., the semialgebraic decomposition of the real binary forms of a given degree's space. We denote $\mathcal{B}_d$ the space of real binary forms of degree $d$, similar to $\mathcal{S}^{\mathbb{K}}_{n}$ or  $\mathcal{S}_{n}$ , used for $\mathbb{K}$ fields in general.  We prove that the sets $\mathcal{W}^{(r)} \subset \mathcal{B}_d$  of real binary forms of real rank $r$ are semialgebraic sets (see Theorem \ref{th-semialg}). Our technique to demonstrate that those sets are all of them semialgebraic is based on Borchardt-Jacobi Theorem (see Theorem \ref{Bez}). The principal minors of bezoutian matrices $B_r(q,q')$ give us a system of conditions which determine the semialgebraic sets. The analogous decomposition in the complex case can be seen in \cite{CS}. Moreover, in order to calculate the dimension of  $\mathcal{W}^{(r)}$  we can use the usual techniques in Real Geometry. This replies, in the real case, to the Q1 question asked by Carlini in \cite{C} for complex binary forms. In fact, for typical rank $r$, the dimension of $\mathcal{W}^{(r)}$ is $d+1$. 
\newpage

As a by-product, in Section 6 we obtain the semialgebraic structure of the set of Waring decompositions of $x^{d-m}y^m$ for $1\leq m\leq d-1$; the monomials are non typical but very interesting forms (see \cite{CCG} for the complex case).
Finally, we include the semialgebraic decomposition for $\mathcal{B}_3$ and $\mathcal{B}_4$ in the Section 7. At the end of this section, when we confront with degrees greater than four, we observe that the description of $\mathcal{W}^{(r)}$ becomes very complicated because of the length and degrees of the polynomials which define it. Therefore we restrict the decomposition for degree 5 to one of the canonical forms that P. Common and G. Ottaviani have described in \cite{CO}. In EACA 2016 \cite{ADZ} we presented the semialgebraic decomposition for one of these canonical forms of degree 5. In \ref{canonical} we compute the semialgebraic decomposition of the second type of canonical form.
\smallskip

There are some questions that remain open. For instance, the dimension of $\mathcal{W}^{(r)}$ for non typical ranks, since for typical ranks the sets $\mathcal{W}^{(r)}$ are semialgebraic sets of maximal dimension.  We are working on this problem for these ranks in fixed  degree  $d$.   B. Reznik \cite{Re2} has also studied canonical forms for polynomials, although he works over $\mathbb{C}$. It is a work in progress the computation of canonical forms for typical Waring ranks.


\section{Preliminaries}

Let be  $\mathcal{B}_d$ the real  space of real binary forms of degree $d$ in the variables $x,y$.
Let be $p(x,y)$ a real binary form in $\mathcal{B}_d$,
\begin{equation}\label{defp}
p(x,y)=p_{\vec{c}}(x,y)= \sum_{i=0}^{d}                         
\left(
\begin{minipage}[c][9pt][b]{8pt}
$$  \begin{array}{c}
   \vspace{-4pt}
   \!\! \scriptstyle d\!\! \\
   \vspace{-4pt}
    \!\! \scriptstyle i  \!\!
  \end{array} $$
\end{minipage}
\right)
c_{i}\, x^{i}\, y^{d-i}, \quad \text{with } \vec{c}=(c_0,\ldots, c_d) \in \mathbb{R}^{d+1}\setminus{\vec{0}}
\end{equation}
A {\sl Waring Decomposition over $\mathbb{R}$ of length $r$ for $p$}  is any rewrite of the form $p$ as a linear combination  of $d$-th powers of linear forms $\ell_i =\alpha_i x+\beta_i y$, $i=1,\dots ,r$, say
\begin{equation}\label{desWaring}
  p(x,y)=\sum_{i=1}^{r}\lambda_i\ell_i^d \ , \quad \text{for some real numbers } \lambda_i \ .
\end{equation}
We also required that this expression is not redundant, that is, $\ell_1 ,\dots ,\ell_r $ are linear independent. The number $r$ is call {\sl the length of the Waring decomposition}. Moreover, if $r$ is the smallest possible length for $p$, we call such $r$ {\sl the real rank of $p$}. 

We associate to each real binary form $p$ a family of Hankel matrices:
\begin{equation}\label{eq:Hr}
H_s=
\left(
    \begin{array}{cccc} c_0&c_1&\cdots &c_s \\ c_1&c_2&\cdots&c_{s+1} \\
    \vdots&\vdots&\ddots&\vdots\\c_{d-s}&c_{d-s+1}&\cdots&c_d
    \end{array}
\right) \quad , \, s=0, \dots , d \ ,
\end{equation}
Their kernels, $Ker(H_s )$, play an essential role in the study of Waring's decomposition. When complex coefficients are considered in this problem, the Sylvester's algorithm rely on the study of these matrices. We include it for the convenience of the reader (see Theorem 2.1 in \citep{BCM} and the refe\-rences therein).

\begin{theorem}[Sylvester's Algorithm] \label{c-SylAlg} A binary form of degree $d$, $p$, can be written as a finite sum of $\ d^{th}$ powers of complex linear forms as (\ref{desWaring}), if and only if
\begin{enumerate}
\item There exists a vector $\vec{q}=(q_0,\cdots,q_r)$ such that $H_r {\vec{q}\;}^t = 0$
\item The form $q(x,y)=\sum_{i=0}^r q_i\,x^i\,y^{r-i}$ factors as a product of $\ r$ distinct complex linear forms, i.e.,
$$
q(x,y)= \prod_{j=1}^r (\beta_j\, x - \alpha_j\, y).
$$
\end{enumerate}
 In that case,
\begin{equation}\label{eq:ComplexSylv}
 p(x,y)=\sum_{i=1}^{r}\lambda_i(\alpha_i x+\beta_i y)^d 
\end{equation}
for  some complex numbers $\lambda_i $, $i=1,\dots ,r$.

\end{theorem}

Our approach to the real Waring decomposition is based in a technical tool to guarantee the existence of real roots for  some polynomial $q(x,y)$ whenever its coefficients are in a linear space $Ker(H_s )$ for some number $s$. We use  the Bezoutian matrix associated to $q(t,1)$ and its derivative $q'(t,1)$, and the Borchardt-Jacobi Theorem.

\begin{definition}\label{def:bezoutian} Let $u(t)=\!\sum_{i=0}^n u_it^i$ and  $v(t)=\!\sum_{i=0}^n v_it^i$ be two real polynomials in a variable $t$ of degree at most $n$. The {\sl Hankel's  Bezoutian } or, simply, {\sl Bezoutian} of $u$ and $v$ is the  matrix
$$
B_n(u,v)=Bez_H(u,v)=(b_{ij})_{1\leq\, i,j\leq\, n},
$$
where the $b_{ij}$ are given by the formula $\dfrac{u(t)v(s)-u(s)v(t)}{t-s}=\sum_{i,j=1}^nb_{ij}\;t^{i-1}s^{j-1}$. Observe that $B_n$ is a symmetric matrix.
\end{definition}

\begin{theorem}[Borchardt-Jacobi Theorem, \cite{Bor}]  \label{Bez} 
The number of distinct real roots of a real polynomial $q(t)$ of degree $r$ is equal to the signature of the matrix $B_r(q,q')$, where $q'$ stands for the  derivative  $\frac{dq}{dt}$.
\end{theorem}

\begin{obs}
We will denote $M_B(i)$ the principal $i^{th}$ minor of the Bezoutian matrix $B_r(q,q')$. Hence \ref{Bez} says that $q$ has $r$ distinct real roots if an only if \,  $M_B(i)>0$, \ for \ $i=1, \dots, r$.
\end{obs}


\section{Real Waring decompositions}

Let fix a real binary form $p(x,y)=p_{\vec{c}}(x,y)= \sum_{i=0}^{d}  \binom{d}{i}    c_{i}\, x^{i}\, y^{d-i}$. In this section we present a procedure to compute a Waring's decomposition of $p$ of length at most $d$. This number is an upper bound for the real rank of $p$. This was proved in  \cite{CO}, Prop. 2.1, but not explicit constructions was given there. This bound seems a lot less polished than Theorem 1.1. in \cite{Re2}, where the length of the Waring decomposition for a binary form is bound by $\ \frac{d+1}{2}\ $ or $\ \frac{d}{2}+1\ $, depending on whether $\ d\ $ is odd or even. But it is important to notice that our statement refers to ``any polynomial" while Sylvester talks about ``a general binary form". Moreover our procedure gives a family of such decompositions. An algorithm (see Algorithm  \ref{alg_WD}) is given to compute a Waring decomposition of length $d$.

\subsection{Real Waring decompostions }

Next, we consider two independent sets of indeterminates over  $\mathbb{R}$, say  $X_0 ,\dots ,X_d$ and $S, S_1  ,\dots ,S_\nu$,  for \break $\nu=(d-1)/2$ if $d$ is odd, and $\nu=d/2-1$ if $d$ is even. We will explain the procedure according to the parity of d.

\subsubsection{Construction for  odd degrees} \label{d_impar}

Let it be $d=2\nu+1$. Take a non zero real binary form $p(x,y)$ as in \ref{defp}, and  $\vec{c} = (c_0 ,\dots, c_d )$ the point of $\mathbb{R}^{d+1}\backslash\{\vec{0}\}$ associated to $p(x,y)$. Now, we consider  the matrix

\begin{equation} \label{V_impar}
    V=
\left(\begin{array}{ccccccc}  X_{0}&1&1&\cdots&1&1&c_{d} \\ X_{1}&S_{1}&-S_{1}&\cdots&S_{\nu}&-S_{\nu}&c_{d-1} \\
\vdots&\vdots&\vdots&\cdots&\vdots&\vdots&\vdots \\
X_{d}&S_{1}^{d}&(-1)^{d}S_{1}^{d}&\cdots&S_{\nu}^{d} &(-1)^{d}S_{\nu}^{d}&c_{0}\end{array}\right) \, ,
\end{equation} 
and we compute its determinant
$$
\det (V)= h(X_0 ,\dots ,X_d )=\Delta_0 X_0 \!+\!\Delta_1 X_1 \!+\!\cdots +\! \Delta_d X_d \in (\mathbb{R} [S_1,\cdots, S_{\nu}])[X_0, \cdots, X_d ].
$$
Hence $\Delta_j$ are polynomials in $\mathbb{R} [S_1,\cdots, S_{\nu}]$. Let us assume that $\Delta_d$ is not the zero polynomial. Then, the real algebraic set $\{\Delta_d = 0 \}$ has an open and dense complementary $\Omega$ in $\mathbb{R}^\nu$. Moreover the rational function $R =-\Delta_{d-1}/\Delta_d \in \mathbb{R}(S_1,\ \cdots, \ S_{\nu}\,)$ is well defined in $\Omega$. Next we define the real algebraic sets in $\mathbb{R}^\nu$:
\newpage
$$
A= \cup_{i=1}^{\nu} \{S_i=0\} \ ,\quad D= \left(\bigcup_{i<j} \{S_i + S_j=0\}\right) \cup \left(\bigcup_{i<j} \{S_i - S_j=0\}\right) \ ,
$$
$$
B= \left(\bigcup_{i=1}^{\nu} \{\Delta_{d-1}+ S_i\Delta_d=0\}\right) \cup \left(\bigcup_{i=1}^{\nu} \{\Delta_{d-1}- S_i\Delta_d=0\}\right)
$$
\smallskip
Then $\mathcal{G} =\Omega\setminus (A\cup B\cup D)$ is an open  semialgebraic set in $\mathbb{R}^\nu$. Moreover $\mathcal{G}$ is non empty, and we can choose $\textbf{s}=(s_1, \cdots, s_{\nu}) \in \mathcal{G}$. Then the real polynomial:
$$
h^* (T)=h(1,T,T^2 ,\dots ,T^d )=\Delta_0 (\textbf{s}) +\Delta_1 (\textbf{s})T +\cdots +\Delta_d(\textbf{s}) T^d  \
$$
has $d=2\nu+1$ real roots: $\pm s_i \in \mathbb{R}\setminus{0}$ and also $R$, which are distinct  by choice.

For ${\bar{c}}^{\,\,t} =(c_d ,\dots , c_0)$ and  ${\vec{\lambda}}^{\,t} =(\lambda_1\, ,\dots ,\,\lambda_d)$, we  consider  the linear system:

\begin{equation}\label{eq:sistema1}
    M\vec{\lambda} =\bar{c} 
\end{equation}
where 
\begin{equation} \label{M_impar}
M=
\left(
\begin{array}{cccccc}  1&1&\cdots&1&1 &1\\ s_{1}&-s_{1}&\cdots&s_{\nu}&-s_{\nu} &R\\
\vdots&\vdots&\cdots&\vdots&\vdots &\vdots\\
s_{1}^{d}&(-1)^{d}s_{1}^{d}&\cdots&s_{\nu}^{d}&(-1)^{d}s_{\nu}^{d} &R^{d}\end{array}
\right).
\end{equation} 
and we  find the wanted Waring's decomposition solving the system \eqref{eq:sistema1}. We point out that $M$ is a   $(d+1) \times \,d$ matrix of  rank  $d$. Also, we have $h^{*}(1,R,R^2,\,\cdots,\,R^d)=0$, and the determinant  $\det (\,M|\,\bar{c}\,)$  equals zero.  Thus, the system \eqref{eq:sistema1} can be solved, and this gives us the solution to the Waring problem in this case. That is,
$$
p(x,y) = \sum_{j=1}^{d} \lambda_j \,L_{j}^{d}\, (x,y),
$$
with $L_j (x,y)=x+s_j\, y$, if $j$ is odd,  $L_j (x,y)=x-s_j \,y$, if $j$ is even, when $j<d$, and $L_d (x,y)=x+R y$.

\medskip
Let us assume that $\Delta_d$ is  the zero polynomial. In this case we consider the following linear system for the fixed $\textbf{s}=(s_1, \cdots, s_{\nu}) \in \mathcal{G}$,

\begin{equation}\label{eq:d-impar-especial}
M\vec{\lambda} =\bar{c} \ , \ \text{ with } \ 
  M=
\left(
\begin{array}{cccccc}  1&1&\cdots&1&1 &0\\ s_{1}&-s_{1}&\cdots&s_{\nu}&-s_{\nu} &0\\
\vdots&\vdots&\cdots&\vdots&\vdots &\vdots\\
s_{1}^{d}&(-1)^{d}s_{1}^{d}&\cdots&s_{\nu}^{d}&(-1)^{d}s_{\nu}^{d} &1\end{array}
\right) , \ \text{and  } \  {\bar{c}}^{\,\,t} =(c_d ,\dots , c_0) \ ,
\end{equation}
to obtain a Waring's decomposition for $p$.

As a consequence of  the given procedure, we rewrite $p$ as a  Waring's decomposition of the form \eqref{desWaring}  for each odd degree $d$.

\begin{obs}
  Observe that for $p=3x^2 y + y^3$, we have $\Delta_3=-2 s_1  (0\cdot s_1^2 - 0) =0$, for any choice of $s_1$. The system \eqref{eq:d-impar-especial} gives
  $$
  \lambda_1 = \frac{1}{2 s_1} ,  \lambda_2 = -\frac{1}{2 s_1} , \lambda_3 = -s_1^2 + 1 .
  $$
 for every $s_1 \not =0$. Then 
 $$
 3x^2 y + y^3 = \frac{1}{2 s_1 } \left( x+s_1 y \right)^3 -
 \frac{1}{2 s_1} \left( x-s_1 y \right)^3 +( -s_1^2 + 1) y^3 \ ,
 $$
and, for $s_1=1$, we obtain a shorter expression since $p$ is a real binary form of real rank $2$, as we will see in the  subsection \ref{subsec:rank}.
\end{obs}

Next, we give an example of the previous procedure.

\begin{ex}  Take $p(x,y)= 212y^5+330xy^4+200x^2y^3+60x^3y^2+10x^4y+x^5$. \label{exp-p5} Choosing $s_1=3$ and $s_2=4$ in the algorithm, we obtain $R=0$ ; secondly $s_1=1$ and $s_2=2$, then $R=3$. Hence, we have
$$
\begin{array}{rl}
p(x,y)&=\displaystyle \frac{11}{21}(x+3y)^5-\frac{1}{21}(x-3y)^5+\frac{5}{56}(x+4y)^5+\frac{1}{56}(x-4y)^5 +\frac{5}{12} x^5=\\
\, \\
  &=(x+y)^5-(x+2y)^5+(x+3y)^5.
\end{array}
$$

\end{ex}

\subsubsection{Construction for  even degrees} \label{d_par}

In this case, $d=2\nu$. Take a non zero real binary form $p(x,y)$ as in \ref{defp}, and  $\vec{c} = (c_0 ,\dots, c_d )$ the point of $\mathbb{R}^{d+1}\backslash\{\vec{0}\}$ associated to $p(x,y)$. In this case we consider  the matrix:

\begin{equation}\label{eq:V_par}
V=
\left(\begin{array}{cccccccc}  X_{0}&1&1&1&\cdots&1&1&c_{d} \\ X_{1}&S&S_{1}&-S_{1}&\cdots&S_{\nu-1}&-S_{\nu-1}&c_{d-1} \\
\vdots&\vdots&\vdots&\vdots&\cdots&\vdots&\vdots&\vdots \\
X_{d}&S^d &S_{1}^{d}&(-1)^{d}S_{1}^{d}&\cdots&S_{\nu-1}^{d} &(-1)^{d}S_{\nu-1}^{d}&c_{0}\end{array}\right)
\end{equation} 
and we compute its determinant
$$
\det (V)= h(X_0 ,\dots ,X_d )=\Delta_0 X_0 \!+\!\Delta_1 X_1 \!+\!\cdots +\! \Delta_d X_d \in (\mathbb{R} [S, S_1,\cdots, S_{\nu-1}])[X_0, \cdots, X_d ].
$$
Now, let consider the polynomial $\Delta_{d}$. First suppose    $\Omega=\mathbb{R}^{\nu}\setminus \{ \Delta_d =0\ \}$ is non empty. Next we define the real algebraic sets in $\mathbb{R}^\nu$:
$$
A= \bigcup_{i=1}^{\nu-1} \{S_i=0\} \quad ,\quad 
B= \left\{
\, \Delta_{d-1}+2 S\Delta_d=0  \ \right\} 
\cup 
\left(\bigcup_{i=1}^{\nu-1} \{\Delta_{d-1}\pm S_i\Delta_d=0\}\right)   \ ,
$$

$$
D= \left(\bigcup_{i<j} \{S_i + S_j=0\}\right) \cup \left(\bigcup_{i<j} \{S_i - S_j=0\}\right) 
\cup \left( \bigcup_{i=1}^{\nu-1} \{S + S_i=0\} \right)  \cup \left(\bigcup_{i=1}^{\nu-1} \{S - S_i=0\} \right)
$$

Then $\mathcal{G} =\Omega\setminus (A\cup B\cup D)$ is an open  semialgebraic set in $\mathbb{R}^\nu$. Moreover $\mathcal{G}$ is non empty, and we can choose $\textbf{s}=(s,s_1, \cdots, s_{\nu-1}) \in \mathcal{G}$. Then the real polynomial:
$$
h^* (T)=h(1,T,T^2 ,\dots ,T^d )=\Delta_0 (\textbf{s}) +\Delta_1 (\textbf{s})T +\cdots +\Delta_d(\textbf{s}) T^d  \
$$
has $d=2\nu$ real roots: $\pm s_i \in \mathbb{R}\setminus{0}$, and also  $s$ and  $R=-\dfrac{\Delta_{d-1}}{\Delta_d}-s$, which are distinct  by choice.

The associated linear system to $p$ is now:

\begin{equation}\label{eq:sistema2}
    M\vec{\lambda} =\bar{c} 
\end{equation}
with $\vec{c} = (c_0 ,\dots, c_d )$ and
\begin{equation}\label{eq:M_par}
M=
\left(\begin{array}{cccccccc} 1&1&1&\cdots&1&1&1\\ s&s_{1}&-s_{1}&\cdots&s_{\nu-1}&-s_{\nu-1} &R\\
\vdots&\vdots&\vdots&\cdots&\vdots&\vdots&\vdots \\
s^d&s_{1}^{d}&(-1)^{d}s_{1}^{d}&\cdots&s_{\nu-1}^{d}& (-1)^{d}s_{\nu-1}^{d}&R^d\end{array}\right)
\end{equation} 
and we  find the wanted Waring's decomposition solving the system \eqref{eq:sistema2}. We point out that $M$ is a   $(d+1) \times \,d$ matrix of  rank  $d$. Also, we have $h^{*}(1,R,R^2,\,\cdots,\,R^d)=0$, and the determinant  $\det (\,M|\,\bar{c}\,)$  equals zero.  Thus, the system \eqref{eq:sistema2} can be solved, and this gives us the solution to the Waring problem in this case. Therefore,
\begin{equation}\label{casopar}
p(x,y) =  \sum_{j=1}^{d} \lambda_j L_{j}^{d} (x,y) \ ,
\end{equation}
with $L_j (x,y)=x+s_j\, y$, if $j$ is even,  $L_j (x,y)=x-s_j\, y$, if $j$ is odd, for $1<j<d$,  $L_1 (x,y)=x+s y$  and  $L_d (x,y)=x+R y$. 

\medskip
Next, let suppose    $\ \Delta_d $ is the zero polynomial. We take the linear system \eqref{eq:sistema2} with the matrix
\begin{equation}\label{eq:MparSpecial}
M=
\left(\begin{array}{cccccccc} 
1&1&1&\cdots&1&1&0
\\ s&s_{1}&-s_{1}&\cdots&s_{\nu-1}&-s_{\nu-1} &0
\\
\vdots&\vdots&\vdots&\cdots&\vdots&\vdots&\vdots 
\\
s^d&s_{1}^{d}&(-1)^{d}s_{1}^{d}&\cdots&s_{\nu-1}^{d}& (-1)^{d}s_{\nu-1}^{d}&1\end{array}\right)    
\end{equation}
Because $\Delta_d$ is the zero polynomial, the system \eqref{eq:sistema2} is solvable and we can obtain $\lambda_1 ,\dots \lambda_d$ such that 
\begin{equation}
   p(x,y) =  \sum_{j=1}^{d-1} \lambda_j L_{j}^{d} (x,y) +\lambda_d y^d\ , 
\end{equation}
with $L_j (x,y)=x+s_j\, y$, if $j$ is even,  $L_j (x,y)=x-s_j\, y$, if $j$ is odd, for $1<j<d$, and $L_1 (x,y)=x+s y$.

\begin{obs} Observe that for $p(x,y)=\dfrac{\varepsilon^2 +1}{\varepsilon} y^4+6 \, \varepsilon x^2 y^2   + 4x^3 y$, the system associate to the matrix \eqref{eq:M_par} gives $ 
  \lambda_1 =0 ,  \lambda_2 = \dfrac{\varepsilon^2 + \varepsilon + 1}{2(\varepsilon + 1)}, \lambda_3 =\dfrac{\varepsilon^2 - \varepsilon + 1}{2(\varepsilon - 1)} ,\lambda_4=-\dfrac{\varepsilon^3 }{\varepsilon^2 - 1}$ ,
 for  $s_1 =1$. Then 
 $$
p(x,y)= \lambda_2  \left( x+ y \right)^4 + \lambda_3 \left( x- y \right)^4 
+ \lambda_4 \left( x- \dfrac{1}{\varepsilon}y \right)^4   \ ,
 $$
and our procedure gives a Waring decomposition of length $3$. This allows  to analyze how  perturbations in the coefficients of the form $p$ are transmitted to  their  Waring's decompositions. 
 
\end{obs}

Next, we give an example of the previous procedure.

\begin{ex}  Take $p(x,y)= 240y^4+224xy^3+72x^2y^2+8x^3y+x^4$ \label{exp-p4}. In this case, we have firstly chosen $s_1=1$ in the algorithm and then $R=\frac{38}{9}$ ; secondly $s_1=2$ and then $R=4$. Hence, we have
$$
\begin{array}{rl}
p(x,y)&=\displaystyle \frac{34}{19}x^4-\frac{40}{29}(x+y)^4-\frac{8}{47}(x-y)^4 +\frac{19683}{25897}\left(x+\frac{38}{9}y\right)^4=\\
\, \\
  &=-(x+2y)^4+(x+4y)^4.
\end{array}
$$
\end{ex}

\subsubsection{The algorithm \texttt{Real Waring Decomposition}}

\medskip

The Waring decomposition constructed in the previous subsections gives a method to exhibit  solutions for the Waring problem for any real binary form of length at most $d$. The linear forms we gave have real coefficients. Observe that if we apply the Sylvester algorithm to $p$, in general there is not guarantee that the linear forms we obtain have real coefficients. This fact is quite delicate and it rely on the fact that $\mathbb{R}$ is a real closed field  in a essential way. Next we propose a method to obtain a Real Waring Decomposition.  In practice, a random choice for \ $ \textbf{s}$ \ it is probably a good input to preforms the proposed procedure.

\medskip

\begin{algorithm}[H]
\caption {Real Waring Decomposition (RWD)}
\label{alg_WD}

   \SetKwInOut{Input}{Input}
    \SetKwInOut{Output}{Output}

\Input  {$p_{\vec{c}\,}(x,y)= \displaystyle \sum_{i=0}^{d}
 \left(
\begin{minipage}[c][9pt][b]{8pt}
$$  \begin{array}{c}
   \vspace{-4pt}
   \!\! \scriptstyle d\!\! \\
   \vspace{-4pt}
    \!\! \scriptstyle i  \!\!
  \end{array} $$
\end{minipage}
\right)
c_{i}\, x^{i}\, y^{d-i}\neq(\alpha x +\beta y)^d$
}
\Output {a real Waring decomposition to $p_{\vec{c}\,}(x,y)$.}

\eIf{$d=2\ell +1$}
    {\vspace{-1mm}
    choose $s_1, \, \ldots,\, s_{\ell}$, real, non zero and distinct numbers \label{lin:s_i} \;
    \vspace{-1mm}
    construct the matrix $V$ as in (\ref{V_impar})\;
    \eIf{$\Delta_d=0$}
    {go to step \ref{lin:s_i}
    \vspace{-1mm}
    }
    { 
    \vspace{-1.2mm}
        determine $R=-\frac{\Delta_{d-1}}{\Delta_d}$\;
        \vspace{-1.2mm}
          \eIf{R is the same as any $s_i$ or their opposite}
          {go to step \ref{lin:s_i}
          \vspace{-1.2mm}
          }
          { 
          \vspace{-1.2mm}
          go to step \ref{lin:M}
          }
          \vspace{-1.2mm}
            construct the matrix $M$ as in (\ref{M_impar}) \label{lin:M}\;
            \vspace{-1.5mm}
            solve the linear system $ M\vec{\lambda} =\bar{c}$\;
            \vspace{-1.2mm}
    }
    \vspace{-1.5mm}
     The $\ $wanted $\ $decomposition is $\ p(x,y) = \sum_{j=1}^{d} \lambda_j \,L_{j}^{d}\, (x,y),\ $ with $L_j (x,y)=x+s_{\frac{j+1}{2}}\, y,$ if $j$ is even, $j<d$, $L_j (x,y)=x-s_{\frac{j}{2}} \,y$, if $j$ is odd, and $L_d (x,y)=x+R y$.
    }
    { 
    \vspace{-1mm}
    $d=2\ell$\;
    \vspace{-1.2mm}
    choose $s,\,s_1, \, \ldots,\, s_{\ell-1}$, real, non zero (except, maybe, s) and distinct numbers \label{lin:s_i_par}\;
    \vspace{-1mm}
    construct the matrix $V$ as in (\ref{eq:V_par})\;
    \vspace{-1.2mm}
    \eIf{$\Delta_d=0$ or $\Delta_{d-1}=0=s$}
    {\vspace{-1mm}
    go to step \ref{lin:s_i_par}
    }
    { 
    \vspace{-1.2mm}
        determine $R=-\frac{\Delta_{d-1}}{\Delta_d}-s$\;
        \vspace{-1mm}
          \eIf{R is the same as any $s_i$, their opposite, or $s$}
          {\vspace{-1mm}
          go to step \ref{lin:s_i_par}
          \vspace{-1mm}
          }
          { 
          \vspace{-1mm}
          go to step \ref{lin:M_par}
          }
          \vspace{-1.5mm}
            construct the matrix $M$ as in \eqref{eq:M_par} \label{lin:M_par}\;
            \vspace{-1.5mm}
            solve the linear system $ M\vec{\lambda} =\bar{c}$\;
            \vspace{-1mm}
    }
    \vspace{-1.5mm}
     The $\ $wanted decomposition is $p(x,y) = \lambda_1\, (x+sy)^d +\sum_{j=2}^{d} \,\lambda_j \,L_{j}^{d}\, (x,y)$,\\ with $L_j (x,y)=x+s_{\frac{j}{2}}\, y$, if $j$ is even, $j<d$, $L_j (x,y)=x-s_{\frac{j-1}{2}} \,y$, if $j$ is odd, and $L_d (x,y)=x+R y$.
    \vspace{-1mm}}
\label{lin:fin}
\end{algorithm}

\bigskip


\subsection{Real Waring rank  decompositions}\label{subsec:rank}

In this section we will show how to compute a real Waring decomposition of minimal length of a real binary form $p(x,y)$. The method we are presenting next is effective although of high computational complexity,  and it points out the importance of Bezoutian matrix analysis in the study of real Waring decompositions. Also we will show how to modify Algorithm 2.1. in \cite{BCM} to get the real rank of a real binary form $p(x,y)$.

\begin{theorem}\label{r-SylAlg}
Let be $p(x,y)$ a real binary form of degree $d$. Then, the following statements are equivalent:
\begin{enumerate}
    \item The form $p$ can be written as a finite sum of $\ d^{th}$ powers of {\bf real} linear forms as 
    \begin{equation}\label{eq:WaringReal1}
  p(x,y)=\sum_{i=1}^{r}\lambda_i \, (\alpha_i x+\beta_i y )^d \ , \quad \text{for some real numbers } \lambda_i , \ \alpha_i , \  \beta_i \ .
\end{equation}
\newpage
    \item There exists a vector $\vec{q}=(q_0,\cdots,q_r)$ such that $H_r {\vec{q}\;}^t = 0$, and the form $q(x,y)=\sum_{i=0}^r q_i\,x^i\,y^{r-i}$ factors as a product of $\ r$ distinct {\bf real} linear forms, in fact,  \begin{equation}\label{eq:WaringReal2}
    q(x,y)= \prod_{j=1}^r (\beta_j\, x - \alpha_j\, y) \ .
    \end{equation}
\end{enumerate}

\end{theorem}

For convenience of the reader we include an elementary proof of this fact in the Appendix (see \ref{Sylvester}). As a consequence of the previous theorem and theorem \ref{Bez}  we have the following corollary.

\begin{corollary}[Real Sylvester's Algorithm] \label{co:SylAlg}
A real binary form of degree $d$, $p$, can be written as a finite sum of $r$ $\ d^{th}$ powers of real linear forms as (\ref{desWaring}),  if
\begin{equation}
\begin{aligned}
 &\text{ There exists a vector } \vec{q}=(q_0,\cdots,q_r) 
   \text{ such that } H_r {\vec{q}\; }^t = 0 \ , \\ 
   &\text{and the Bezoutian matrix of the polynomial  } q(t)=\sum_{i=0}^r q_i\,t^{r-i}\  \text{  is positive defined.}    
\end{aligned}
\end{equation}

 Moreover, if $q(x,y)= \prod_{j=1}^r (\beta_j\, x - \alpha_j\, y).$, with $\alpha_j$ and $\beta_j$ reals, then the form $p$ can be rewrite as:

\begin{equation}\label{eq:RealSylv}
 p(x,y)=\sum_{i=1}^{r}\lambda_i\ell_i^d 
\end{equation}
for $\ell_i =\alpha_i x+\beta_i y$ and  some real numbers $\lambda_i $, $i=1,\dots ,r$.

\end{corollary}

\begin{ex} \label{ex_rank} Let be $p(x,y)= y^5+\frac{1}{2}x^2y^3-\frac{1}{2}x^4y$.
  Now, we are going to use the Algorithm \ref{alg_rank} to determine a Waring decomposition of length the rank of this polynomial.
\begin{enumerate}
  \item Compute the kernel of  $H_1$ and  the kernel of  $H_2$, where
  $$
  H_1=
  \left(
  \begin{array}{cc}
    1 & 0 \\
    0 & \frac{1}{20} \\
    \frac{1}{20} & 0 \\
    0 & \frac{-1}{10} \\
    \frac{-1}{10} & 0
  \end{array}
\right) 
\quad \text{ and } \quad
H_2=
  \left(
  \begin{array}{ccc}
    1 & 0& \frac{1}{20} \\
    0&\frac{1}{20} & 0 \\
    \frac{1}{20}& 0 & \frac{-1}{10} \\
    0&\frac{-1}{10} & 0
  \end{array}
\right)
$$
Since $ Ker(H_1)=\{\underline{0}\,\}$ and $Ker(H_2)=\{\underline{0}\,\}$, we must compute the kernel of  $H_3$:
$$
H_3=
   \left(
  \begin{array}{cccc}
    1 & 0 &\frac{1}{20}&0 \\
    0&\frac{1}{20} & 0& \frac{-1}{10} \\
    \frac{1}{20}& 0& \frac{-1}{10} &  0
  \end{array}
\right)
$$  

  \item Compute a basis of $ Ker(H_3)$, for instance $ (0,2,0,1)$. This vector can be associated with $q(t)=2 t^2+1$, with two distinct roots in $\mathbb{C}$, but not in $\mathbb{R}$.  Therefore, the real rank it can not be 3, but the complex rank is 3 and we can write:
  $$
  p(x,y) = \frac{41}{40}y^5-\frac{1}{80}\left( y+i\sqrt{2}x \right)^5- \frac{1}{80}\left( y-i\sqrt{2}x \right)^5
  $$
    \item Next compute the kernel of  $H_4$, where
  $$
  H_4=
   \left(
  \begin{array}{ccccc}
    1 & 0 &\frac{1}{20}&0& \frac{-1}{10} \\
    0&\frac{1}{20} & 0& \frac{-1}{10} &  0
  \end{array}
\right)
$$
Then $ Ker(H_4)$ is generated by the set
$ \left\{  \left(1,0,0,0,10\right), (0,1,0,1/2,0),\left(0,0,1,0,1/2 \right) \right\} $. 
For a generic vector in this kernel $(1,\lambda_1,\lambda_2,\lambda_1/2,10+\lambda_2)$,  its Bezoutian matrix is
\newpage
$$
\left(
\begin{array}{cccc}
  \frac{1}{4} \lambda_1^2-\lambda_2^2 -20\lambda_2 &  -\lambda_1(\lambda_2+30) & \frac{1}{2}\lambda_1^2 -2\lambda_2 -40  & \frac{1}{2}\lambda_1 \vspace{1mm}\\ 
   -\lambda_1(\lambda_2+30) &  -\lambda_1^2+2\lambda_2^2 -2\lambda_2-40 &-\frac{3}{2}\lambda_1+2\lambda_1\lambda_2 & 2\lambda_2 \vspace{1mm}\\ \frac{1}{2}\lambda_1^2 -2\lambda_2 -40  & -\frac{3}{2}\lambda_1+2\lambda_1\lambda_2 & 3\lambda_1^2-2\lambda_2 & 3\lambda_1 \vspace{1mm}\\ 
   \frac{1}{2}\lambda_1  & 2\lambda_2  & 3\lambda_1 &4
\end{array}
\right)
$$
so that $q$ will have 4 different real root if the $\lambda_i$ check simultaneously
$$
\begin{array}{l}
\frac{1}{4} \lambda_1^2-\lambda_2^2 -20\lambda_2>0 \\
-(1/4)\lambda_1^4+(1/2)\lambda_1^2\lambda_2^2-2\lambda_2^4-(81/2)\lambda_1^2\lambda_2-38\lambda_2^3-910\lambda_1^2+80\lambda_2^2+800\lambda_2>0 \\
-(1/2)\lambda_1^6-2\lambda_1^4\lambda_2^2-2\lambda_1^2\lambda_2^4-(359/2)\lambda_1^4\lambda_2-29\lambda_1^2\lambda_2^3+4\lambda_2^5-(43449/16)\lambda_1^4+\\
+(2709/4)\lambda_1^2\lambda_2^2+68\lambda_2^4+8765\lambda_1^2\lambda_2-472\lambda_2^3-3600\lambda_1^2-4320\lambda_2^2+9600\lambda_2+64000>0\\
-(1/2)\lambda_1^6-2\lambda_1^4\lambda_2^2-2\lambda_1^2\lambda_2^4-(357/2)\lambda_1^4\lambda_2-25\lambda_1^2\lambda_2^3+8\lambda_2^5-(43467/16)\lambda_1^4+\\
+1034\lambda_1^2\lambda_2^2+128\lambda_2^4+13800\lambda_1^2\lambda_2-1248\lambda_2^3-9600\lambda_1^2-10880\lambda_2^2+38400\lambda_2+256000>0\end{array}
$$

In particular, when $\lambda_1=0$, the five inequalities are verified for $-20<\lambda_2<1-\sqrt{41}$.

For example, if we take $\lambda_2=-52/9$, the vector $(1,0,-52/9,0,64/9)$ corresponds to $q(t)=(t-2)(t+2)(t-4/3)(t+4/3)$. Therefore, its real rank is $4$.

\item Solve the associate linear system $M\vec{\lambda}=\vec{c}$, where  the matrix $M$ is defined as:

$$
 M=
  \left(\begin{array}{cccc}  1&1&1&1\\ 
  2 &-2&\frac{4}{3}&-\frac{4}{3} \vspace{1mm}\\
4 &4&\frac{16}{9}&\frac{16}{9}  \vspace{1mm}\\
8 &-8&\frac{64}{27}&-\frac{64}{27}  \vspace{1mm}\\
16 &16&\frac{256}{81}&\frac{256}{81}  \vspace{1mm}\\
32 &-32&\frac{1024}{243}&-\frac{1024}{243}  \vspace{1mm}\\
\end{array}\right) .
$$

\item Then, a real rank Waring decomposition for $p$ is
$$
 p(x,y)=\frac{41}{1600}(x+2y)^5-\frac{41}{1600}(x-2y)^5-\frac{243}{3200} \left(x+ \frac{4}{3}y\right)^5 +\frac{243}{3200} \left(x- \frac{4}{3}y\right)^5.
$$
\end{enumerate}
\end{ex}

\begin{obs} We usually take the polynomial $q(x,y)$ of Corollary \ref{co:SylAlg} dehomogenized  with $t=\frac{y}{x}$ or $t=\frac{x}{y}$ and this makes easier the factorization. Nevertheless, there exists exceptional forms, as $p(x,y)=x^{2\nu}-y^{2\nu}$, which kernel's polynomial for $H_2$ loses its different real roots if we consider $q$ associated with the vector $(0,1,0)$ as $q(t)=t$ instead of $q(x,y)=xy$.
\end{obs}

\subsubsection{The algorithm \texttt{Real Waring Rank Decomposition}}\label{subsec:RRD}

Next we present the procedure to compute the Waring decomposition of minimum length. We call it the RRD decomposition. The key point will be to use Theorem \ref{Bez} to guaranty the existence of a polynomial $q(t)$ of degree $r$ associated to the kernel of the Hankel matrix $H_r$ such that $q$ has $r$ different real roots.

\medskip

\begin{algorithm}[H]
\caption {Real Rank Length's Decomposition (RRD)}
\label{alg_rank}

   \SetKwInOut{Input}{Input}
    \SetKwInOut{Output}{Output}

\Input  {$p_{\vec{c}\,}(x,y)= \displaystyle \sum_{i=0}^{d}
\left(
\begin{minipage}[c][9pt][b]{8pt}
$$  \begin{array}{c}
   \vspace{-4pt}
   \!\! \scriptstyle d\!\! \\
   \vspace{-4pt}
    \!\! \scriptstyle i  \!\!
  \end{array} $$
\end{minipage}
\right)
c_{i}\, x^{i}\, y^{d-i}$ or its asso\-ciated point $\underline{c}$.
}
\Output {a Waring decomposition to $p_{\vec{c}\,}(x,y)$ with minimal length.}

Initialize $r=1$\;
Define $H_r$ as (\ref{eq:Hr}) and determine its kernel:
$
H_r=<v_1,\,\ldots,\,v_{\delta_r}>
$\;
\eIf{$KerH_r=\{0\}$}
    {
    increment $r \leftarrow r+1$ and go to step 2.
    }
    { 
    define
    $\vec{q}=\vec{q}(\mu_1,\,\cdots,\mu_{\delta_r})= (q_0,\,\ldots \,q_r)= \sum_{i=1}^{{\delta_r}}\mu_i\,v_i$, a kernel's vector\;
    \eIf {$q_0\neq 0$}
    {
    consider $q(t)=\sum_{i=0}^{r}q_i\,t^{r-i}$ \label{lin:q}\;
    calculate $B_r(q,q')$ \label{lin:bez}\;
        \eIf {it is possible to find $(\mu_1^{\star},\,\cdots,\mu_{\delta_r}^{\star})\in \mathbb{R}^{\delta_r}$ such that $B_r(q,q')$ is positive definite \label{lin:bez_def_pos}
        }
        {
        factorize $q(t)=\prod_{i=1}^{r}(t-\alpha_i)$\;
        solve the linear system
        \vspace{-2mm}
        $$
        \left(\begin{array}{cccc} 1&1&\cdots&1\\
        \alpha_1 &\alpha_2&\cdots&\alpha_r\\
        \alpha_{1}^{2}&\alpha_2^{2}&\cdots&\alpha_{r}^{2}\\
        \vdots&\vdots&\cdots&\vdots \\
        \alpha_{1}^{d}&\alpha_2^{d}&\cdots&\alpha_{r}^{d}
        \end{array}\right)
        \left(\begin{array}{c}\lambda_1\\ \lambda_2 \\ \vdots \\ \lambda_r \end{array} \right) =
        \left(\begin{array}{c}c_d\\c_{d-1} \\ \vdots \\ c_0\end{array} \right)
        $$
        \vspace{-2mm}
        } 
        {
        increment $r \leftarrow r+1$ and go to step 2\;
        } 
    } 
    {
    take $q(t)=\sum_{i=0}^{r}q_i\,t^{i}$ and go to step \ref{lin:bez}\;
    } 
} 
    The wanted decomposition is
    \vspace{-2mm}
    $$
\begin{array}{cl}
  p(x,y)=\sum_{i=1}^{r}\lambda_i\left(x-\alpha_i\,y\right)^d & \textrm{if} \ q\  \textrm{was defined in the step \ref{lin:q}}, \\
  p(x,y)=\sum_{i=1}^{r}\lambda_i\left(y-\alpha_i\,x\right)^d & \textrm{in another case}.
\end{array}
$$
\vspace{-3mm}
\end{algorithm}

\subsubsection{The monomials}

\medskip

It is known that the real rank of a non trivial degree $d$  monomial in two variables (trivially, the monomials $x^d$ or $y^d$ have rank 1) is d (see \cite{BCG}). In the complex case (see \cite{CCG}), the rank is given by the expression  $rk(x_1^{a_1}x_2^{a_2})=a_2+1,\ $ for $ 1\leq a_1\leq a_2$.

We review first this fact from our approach by using Bezoutians. Let consider a monomial $x^m\,y^{d-m}$, with $m\geq 1$. We can assume $\ m\leq \frac{d}{2}$ in order to study its Waring decompositions. Following notations \ref{defp}, we rewrite $x^m\,y^{d-m} $ as
$$p_{\underline{c}}=\ p(x,y)=  \left (\! \!
\begin{array}{c}
  d \\
  m
\end{array}
\! \! \right)
c_m \, x^m \, y^{d-m} \quad \text{with } \underline{c}=(\underbrace{0,0,\ldots,0}_{m-1},c_m,0, \ldots,0) \ .
$$
Let $\ell>0$. The corresponding Hankel matrix for the rank of $d-\ell$ is
$$
H_{d-\ell}=
\left(
    \begin{array}{cccccccc} 0&\cdots& 0&c_m& 0&\cdots &0 \\
    0&\cdots &c_m& 0&0&\cdots &0 \\
    \vdots& \iddots & \vdots&\vdots&\vdots&\cdots&0\\
    c_m& \cdots & 0&0&0&\cdots&0
    \end{array}
\right)
$$
whose  kernel's vectors are $(0, \, \ldots, \, 0,q_{m+1}, \, \ldots \, ,q_{d-\ell},0, \, \ldots \, ,0)\,$, and we can write the corres\-ponding polynomials as $\quad \sum_{i=m+1}^{d-\ell}q_i t^{d-\ell-i}\,$. But, both polynomials does not have $d-\ell$ real different roots (see  \cite{BCG}, Lemma 4.1). Next we  compute some explicit examples.

\begin{exs}\label{ex:monomios} 

Monomial $x^{d-2}y^2$ for degrees $d=4$ and  $d=5$.

\begin{enumerate}
   
\item  {\rm The monomial $x^2y^2$}. The complex rank of this monomial is 3, and we can write:
$$
x^2y^2\!=\!\frac{1}{72}(x+2y)^4-\frac{1-i\sqrt{3}}{144}\left(x+(-1+i\sqrt{3})\,y\right)^4- \frac{1+i\sqrt{3}}{144}\left(x-(1+i\sqrt{3})\,y\right)^4.
$$
However, its real rank is 4 and a decomposition is
$$
x^2y^2=\frac{1}{4}(x+y)^4+\frac{7}{108}(x-y)^4-\frac{1}{54}(x+2y)^4 -\frac{8}{27}\left(x+\frac{1}{2}y\right)^4.
$$
Moreover, we can find polynomials as near as we want with minor rank. Take for $m>1$:
$$
  p_m(x,y)=\frac{1}{m}y^4+x^2y^2 =
  -\frac{1}{36}x^4+ \frac{m}{72}\left(x-\frac{\sqrt{6m}}{m}y\right)^4+  \frac{m}{72}\left(x+\frac{\sqrt{6m}}{m}y\right)^4 .
$$

\item {\rm The monomial $x^3y^2$}. Although the rank of a \emph{general} binary form of degree 5, according to Sylvester, is 3, Proposition 3.1 in \cite{CCG} says us that the complex rank for this monomial is 4. For example, we can write

$$
  x^3y^2=
  \frac{1}{40}(x-y)^5+\frac{1}{40}(x+y)^5-\frac{1}{40}(x-iy)^5 -\frac{1}{40}(x+iy)^5.
$$
However, its real rank is 5. Running the Algorithm \ref{alg_WD}, we can find the next family of decompositions for the monomial:

$$
\begin{aligned}
x^3y^2 = &\dfrac{b^2}{20a^2 (b^2-a^2)}\, (x+a y)^5+ \dfrac{b^2}{20a^2 (b^2-a^2)}\, (x-a y)^5+
\\
\  & +\dfrac{a^2}{20b^2 (a^2-b^2)}\, (x+b y)^5 + \dfrac{a^2}{20b^2 (a^2-b^2)}\,(x-b y)^5- 
   \dfrac{a^2+b^2}{10(a^2 b^2)} x^5
\end{aligned}
$$
depending on two parameters and well defined for $a, \ b$  non zero real parameters such that $a \neq \pm b$. However, we can find binary forms with smaller rank as near the monomial as we want. For example, in the coefficients space of the polynomials of $5^{th}$ degree, the  polynomial $x^3y^2+\frac{1}{m}\,xy^4$ belongs to any open ball centered at the monomial $x^3y^2$. But, in fact, its rank is $3$ ($Ker(H_2)=\{\underline{0}\}$)  and we have:
$$
  x^3y^2+\frac{1}{m}\,xy^4 = - \frac{m}{20}x^5
+\frac{m}{40}\left (x+\frac{\sqrt{2m}}{m}\,y \right)^5 +\frac{m}{40}\left (x-\frac{\sqrt{2m}}{m}\,y \right)^5 \ , \ \text{for } m>0.
$$
\end{enumerate}

\end{exs}


\section{Semialgebraic decomposition of the space of real binary forms of degree $d$}\label{section:semialgebraicos}

In this section we give a procedure to compute the real rank of any real binary form of degree $d$. This problem has a projective nature since $p$ and $\lambda p$ have the same real rank, for any $\lambda\in \mathbb{R}$. In the subsection \ref{subsec:RRD} we propose an algorithm to compute the real rank of any binary form; this algoritm gives an explicit description of the set 
$\mathcal{W}^{(r)}$ of real binary forms of real rank $r$. For complex binary forms the set of binary forms of complex  rank $r$ is a constructible set (see \cite{CS}). In the real case, these sets turns out to be semialgebraic sets. This property can be deduce also  from \cite{QCL}, but no explicit description was given there.  In the following we will give an explicit semi-algebraic description of such sets. Some examples of this decomposition are included in the last section of this paper for low degrees.

\medskip

\subsection{Semialgebraic decompositions}\label{descGlobal}

Let $V$ a real vector space of  dimension $d+1$. We write $\mathbb{P}(V)$ the projective space over the real vector space $V$. This real algebraic manifold has real dimension $d$ as a real algebraic variety. Moreover, if $V=\mathbb{R}^{d+1}$  we write $\mathbb{P}^d$ for shorten. Each point $\underline{c}$ in $\mathbb{P}(V)$ can be expressed in homogeneous coordinates $\underline{x}=[x_0:x_1:\cdots:x_d]$ once a basis in $V$ is fixed.

Now, we consider the real algebraic sets in $\mathbb{P}^d \times \mathbb{P}^d$
\begin{equation}\label{eq:acople}
 \mathcal{A}^{(r)}=\left \{(\underline{x},\, \underline{y}) \in \mathbb{P}^d \times \mathbb{P}^d \  : \ 
x_{\nu} y_0 + x_{\nu +1} y_1 +\cdots +    x_{\nu +r} y_r =\ 0 \ 
,\ 0\leq \nu \leq d-r \ \right\}.   
\end{equation}
for $r=1, 2, \dots ,d$. Let consider $\mathcal{H}_{r+1} =\mathbb{P}^d \times \left\{ \underline{y} \in \mathbb{P}^d \ : \  y_{r+1}=0\right\} $. It is easy to verify that 
\begin{equation}
  \mathcal{A}^{(r)} \cap \mathcal{H}_{r+1} \subset \mathcal{A}^{(r+1)} \, . 
\end{equation}

Next we proceed to describe the real algebraic sets where "real polynomials have all their roots real". By Borchardt-Jacobi Theorem, \cite{Bor}, these sets are described by the principal minors of Bezoutian matrices. For instance, for degree $3$, if we take the polynomial $q(t)=q_0 t^3+q_1 t^2 +q_2 t +q_3$, its associated Bezoutian matrix is
$$
B_3 =\begin{pmatrix} q_2^2 - 2 q_1 q_3 & q_1 q_2 - 3 q_0 q_3  & q_0 q_2 \\ 
q_1 q_2 - 3 q_0 q_3 & 2 q_1^2 - 2 q_0 q_2 & 2 q_0 q_1
\\
q_0 q_2 & 2 q_0 q_1 & 3 q_0^2
\end{pmatrix}
$$
Its principal minors are
\begin{equation}\label{eq:MB3}
\begin{aligned}
M_B (1)= & q_2^2 - 2 q_1 q_3
\ , \ 
M_B (2)=q_1^2 q_2^2 - 2 q_0 q_2^3 - 9 q_0^2 q_3^2 - 2 (2 q_1^3 - 5 q_0 q_1 q_2) \  q_3
\ , 
\\
 \ & M_B (3)= q_0^2 q_1^2 q_2^2 - 4 q_0^3 q_2^3 - 27 q_0^4 q_3^2 - 2 (2 q_0^2 q_1^3 - 9 q_0^3 q_1 q_2) \ q_3 \ .
\end{aligned}
\end{equation}
Hence, $q$ has three distinct real roots if and only if  $M_B (1)>0 , M_B (2)>0 , M_B (31)>0$. In particular, we recover the well known conditions for monic cubic polynomials  with $q_1 =0$:
$$
q_2^2 >0 
\ , \ 
-2 q_2^3 - 9 q_3^2 >0 
\ , \ 
-4 q_2^3 - 27 q_3^2 >0 \ .
$$
In general, we consider the semialgebraic sets in $\mathbb{P}^d$ defined by the positivity of the Bezoutian matrix $B_r$ for each $r$ in $\{ 1, \dots , d\}$. Let define 

\begin{equation}\label{eq:positividad}
  \mathcal{S}^{(r)}=\left \{
  [q_0 : \dots : q_r :0 : \dots : 0] \in \mathbb{P}^d 
  \ : \ 
  \ M_{B_r } (i)>0,\ \textrm{for} \ i = 1, \dots ,r  \right\} \subset \mathbb{P}^d.  
\end{equation}

Finally, we must consider the condition given in \eqref{eq:acople} and also \eqref{eq:positividad}, and then we define the global semialgebraic sets in $ \mathbb{P}^d \times  \mathbb{P}^d $:

\begin{equation}\label{def-Fr}
 \mathcal{F}^{(r)}= \mathcal{A}^{(r)} \cap \left(\mathbb{P}^d \times \mathcal{S}^{(r)} \right) \subset \mathbb{P}^d \times \mathbb{P}^d   
\end{equation}
The set  $\mathcal{F}^{(r)}$ is an intersection of a real  algebraic set with a semialgebraic. Hence, it is a semialgebraic set. We will point out this fact in the  theorem \ref{prop:Fr-semi}. It is to be noted that these sets encode the binary real forms and their  Waring decompositions. 

\begin{proposition} \label{prop:Fr-semi} For each $r$ in $\{ 1, \dots , d\}$, the set 
 $\mathcal{F}^{(r)}$ is a semialgebraic set of $\mathbb{P}^{d}\times \mathbb{P}^{d} $. Moreover,   if we consider  the projection $\pi : \mathbb{P}^{d}\times \mathbb{P}^{d} \rightarrow \mathbb{P}^{d}$ given by $\pi (\underline{x},\underline{y}) =\underline{x}  $, then the set 
 \begin{equation}\label{conj- a lo mas r}
  \mathcal{E}^{(r)}=\pi \left(\mathcal{F}^{(r)}\right)    
 \end{equation}
is a semialgebraic set.
\end{proposition}

For each $r$ in $\{ 1, \dots , d\}$, the set $\mathcal{E}^{(r)}$ describe the set of real binary form of real Waring rank at most $r$.

Moreover, by the Theorem of the Complementary for global semialgebraic sets (see \cite{BCR} ), we obtain that the set $\mathcal{E}^{(r)} \setminus{\mathcal{E}^{(r-1)}}$ is semialgebraic.  In fact, we have the following result.

\begin{theorem}\label{th-semialg}
  
Let be $r$ in $\{ 1, \dots , d\}$. The subset of $\mathbb{P}^{d}$ given by 
\begin{equation}\label{conj-exacto r}
 \mathcal{W}^{(r)} =\mathcal{E}^{(r)} \setminus{\mathcal{E}^{(r-1)}}\quad , \  \text{ with } \mathcal{E}^{(0)}=\emptyset \ ,
\end{equation}
is a semialgebraic set; it describes the set of real binary forms of real Waring rank $r$. As a consequence, this set is a disjoint union of a finite number of connected semialgebraic   sets.

\newpage
Moreover, we have the semialgebraic decomposition of the real vector space of real binary forms of degree $d$:

\begin{equation}\label{eq:PB}
\mathbb{P}(\mathcal{B}_d ) =\bigcup_{r=1}^{d}  \mathcal{W}^{(r)} \ .    
\end{equation}

\end{theorem}
We include in this section a basic example to show the decomposition \eqref{eq:PB} for $\mathbb{P}(\mathcal{B}_3 )$  (see example \ref{ex:B3}). The computation of the decomposition \eqref{eq:PB} for $\mathbb{P}(\mathcal{B}_4 )$ is pretty complicated and it is included in  section \ref{sect:ex}.

\begin{proposition} Let be  $\mathcal{B}_d$ the real vector space of real binary forms of degree $d$ in the variables $x,y$. Let be $\mathbb{P}(\mathcal{B}_d )$ the projective space over the real vector space $\mathcal{B}_d$. The function {\sl real rank}:
$$
\rho_d : \mathbb{P}(\mathcal{B}_d ) \rightarrow \{ 1, \dots ,d \} 
\quad , \quad \rho_d ([c_0 : \dots :c_d ]) = rk_{\mathbb{R}}(p_{\vec{c}})
$$
is a semialgebraic function.
\end{proposition}

\begin{proof} Observe that the set
$$
\left\{
(\underline{c} , rk_{\mathbb{R}}(p_{\vec{c}}) )  \ : \ p_{\vec{c}} \in \mathcal{B}_d \ , \ p_{\vec{c}} \not= 0
\right\} =
\bigcup_{r=1}^{d}  \mathcal{W}^{(r)} \times \{r\}
$$
describes the graph of $\rho_d$. So, it is semialgebraic.
\end{proof}

\begin{ex}\label{ex:B3} Let be  $\mathcal{B}_3$ the real vector space of real binary forms of degree $3$ in the variables $x,y$. Let be $\mathbb{P}(\mathcal{B}_3 )$ the projective space over the real vector space $\mathcal{B}_3$. Let us decompose $\mathbb{P}(\mathcal{B}_3 )$ by means of the real rank function for real binary forms.

By direct computation, we obtain that $\mathcal{W}^{(1)}$ is the projective curve given by

$$
 \mathcal{W}^{(1)} =
\left\{
[1:\alpha:\alpha^2:\alpha^3]  
\ | \ \alpha \in  \mathbb{R} \right\} 
\cup 
\left\{  ^{\ }
[0 :0 :0:1] \   \right\}.
$$

\begin{figure}[ht]
\centering
\includegraphics[width=0.45 \textwidth]{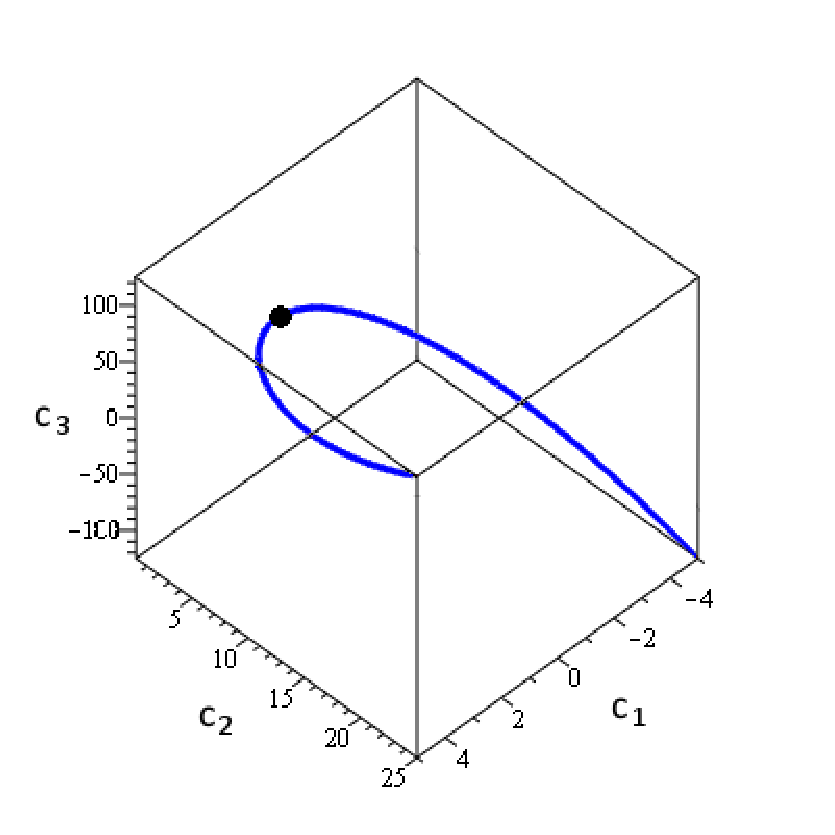}
\caption{Graphic of $\mathcal{W}^{(1)}$ when $c_0=1$. Market the monomial $y^3$.}
\label{fig:B3W1}
\end{figure}
(see figure \ref{fig:B3W1}). In order to determine $\mathcal{W}^{(2)} =\mathcal{E}^{(2)} \setminus\mathcal{W}^{(1)}$, we will analyze $\mathcal{E}^{(2)} = \pi \left(\mathcal{F}^{(2)}\right)    $, where $\pi : \mathbb{P}^{3}\times \mathbb{P}^{3} \rightarrow \mathbb{P}^{3}$ given by $\pi (\underline{x},\underline{y}) =\underline{x}  $ and 
$$
\mathcal{F}^{(2)} = \left \{(\underline{c},\, \underline{q}) \in \mathbb{P}^3 \times \mathbb{P}^3 \  : \ 
c_{0} q_0 + c_{1} q_1 +    c_{2} q_2 =\ 0 \ 
,c_{1} q_0 + c_{2} q_1 +   c_{3} q_2 =\ 0 \ , M_B (1)>0 , M_B (2)>0  \ \right\}.
$$
where $M_B (i)$ is the $i^{th}$ principal minor of the Bezoutian matrix defided in \ref{def:bezoutian} for the polinomial $q_0 t^2 +q_1  t +q_2$. First observe that $\mathcal{F}^{(2)} \cap \{ \ q_0 =0 \ \} =\emptyset $ and $\mathcal{F}^{(2)} \cap \{ \ q_2 =0 \ \} \subset \pi^{-1} ( \mathcal{W}^{(1)} ) $. Next we proceed to find the inequalities defining $\mathcal{W}^{(2)} $. For this, we consider
$$
Q_1 = c_0c_2-c_1^2 \ , \ 
Q_2 = c_1c_3-c_2^2\, , \quad \textrm{ and }   \ Q =Q_1^2 +Q_2^2 \ .
$$
Then, whenever $Q\not=0$ the linear system 
$$
c_{0} q_0 + c_{1} q_1 +    c_{2} q_2 =\ 0 \ 
,  \  c_{1} q_0 + c_{2} q_1 +   c_{3} q_2 =\ 0 \ ,
$$
can be solved for $\underline{c} \not \in    \mathcal{W}^{(1)} $, and we obtain:
\begin{equation}\label{LasQi}
q_0 = \frac{c_2^2 q_2 - c_1 c_3 q_2}{c_1^2 - c_0 c_2} \ , \ 
q_1 =-\frac{c_1 c_2 q_2 - c_0 c_3 q_2}{c_1^2 - c_0 c_2 } \ \textrm{ or } \ 
q_1 = -\frac{c_1 c_2 q_0 - c_0 c_3 q_0}{c_2^2 - c_1 c_3} \ , \ 
q_2 = \frac{c_1^2 q_0 - c_0 c_2 q_0}{c_2^2 - c_1 c_3} \ ,
\end{equation}
and then we replace these expressions in the Bezoutian principal minors'  inequatilies
$
M_B (1)>0 , M_B (2)>0  \ .
$
Next, computing with the equations \eqref{LasQi} in these inequalities, we obtain that 
$$
 \mathcal{W}^{(2)} =
\left\{ \  [c_0 : c_1 : c_2 : c_3] \in \mathbb{P}^3 \ : \ 
 \ 
Q_1  Q_2\not= 0 \ ,\   f >0 \ , \  f+2Q_1 Q_2 >0
\right\} 
$$
where $f$  is the homogeneous polynomial

\begin{equation}
 f(c_0,c_1,c_2,c_3)=
f(c_0,c_1,c_2,c_3)=c_0^2c_3^2 -6c_0c_1c_2c_3 +4c_0c_2^3 + 4c_1^3c_3-3c_1^2c_2^2  \ .
\end{equation}

Let decompose  $ \mathbb{P}^3 =U_0 \cup H_\infty$, with $U_0= \{\ c_0 \not=0 \ \}$ and $H_\infty = \{ c_0 = 0 \}$ . Then, we have $\mathcal{W}^{(2)} =X_0^{(2)}  \cup X_{0,\infty}^{(2)} $, where $ X_0^{(2)}= \mathcal{W}^{(2)} \cap U_0 $ and $ X_{0,\infty}^{(2)}=\mathcal{W}^{(2)} \cap  H_\infty $. The border $\partial X_0^{(2)} \subset \mathbb{R}^3 $ can be plotted using the program Surfer. We include its graphic for convenience of the reader. On the left graphic of Figure \ref{fig:doble}, $X_0^{(2)}$ is limited by the light surface. Observe that the origin of the right graphic in Figure \ref{fig:doble} corresponds to the monomial $y^3$. The remaining monomials must be looking for in $\partial X_{0,\infty}^{(2)}$.

Finally, we get
$\mathcal{W}^{(3)}$ as the complementary of $\mathcal{W}^{(1)}\cup \mathcal{W}^{(2)}$. So, $\mathcal{W}^{(3)}$ is also a real semialgebraic set.

\end{ex}

\begin{figure}[ht]
  \centering
    \includegraphics[width=0.9 \textwidth]{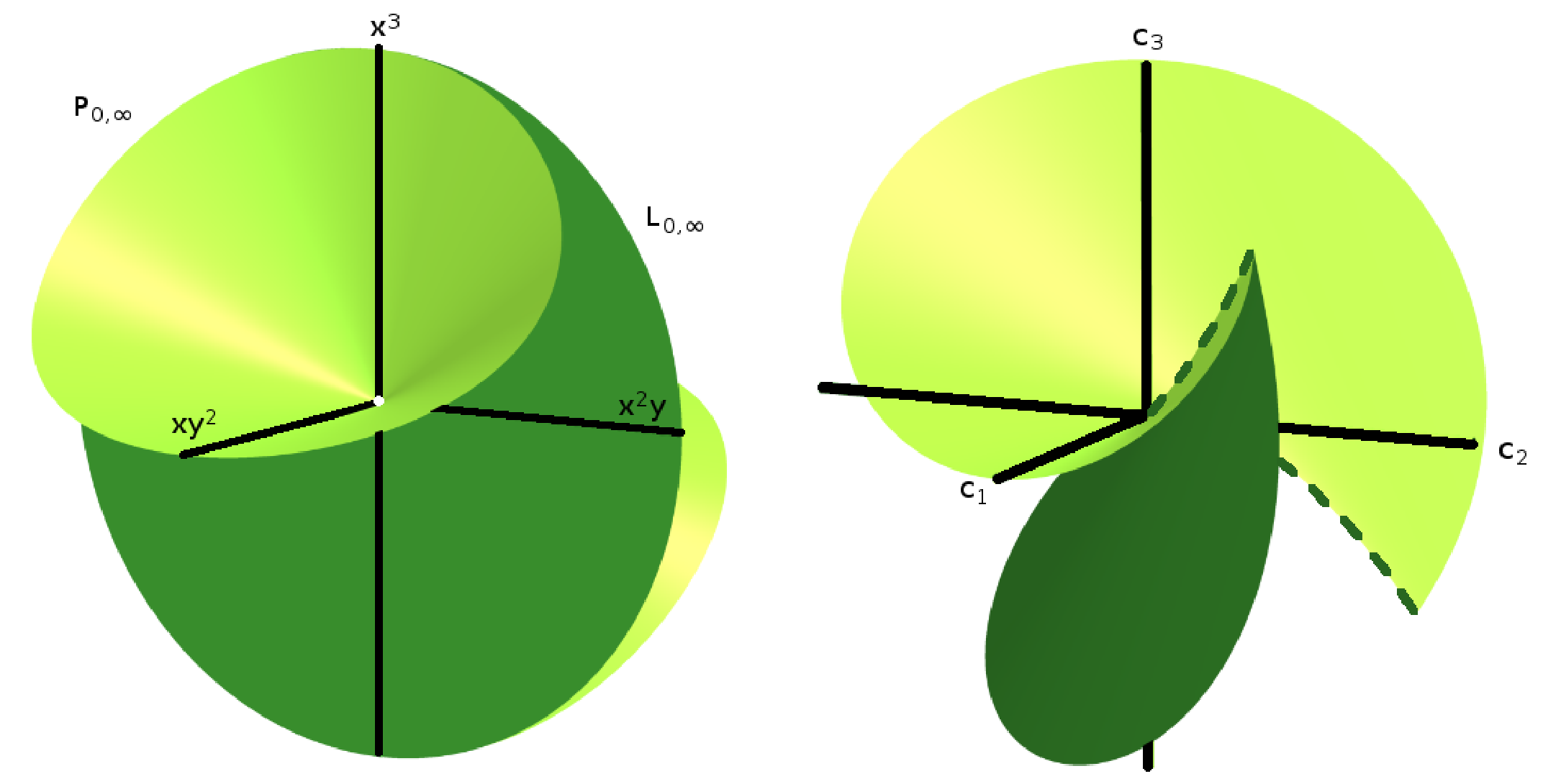}
  \caption{LHS, graphic of $\partial X_{0,\infty}^{(2)}$. RHS, $\partial X_0^{(2)}$. Marked $\mathcal{W}^{(1)}$ with dashed line.}
  \label{fig:doble}
\end{figure}

\medskip

\subsubsection{The typical ranks' strata}

 In \cite{Ble} and  \cite{CO}, for example, a {\it typical rank } $r$ is a rank such that $\mathcal{B}_d$ contains a non-empty open set  of real binary forms of rank $r$ (for  the usual topology of $\mathbb{R}^{d+1}$). We are going to work with $p$, called {\it typical real form} of real rank $r$   if there exits an open neighborhood of $p$ (in $\mathcal{B}_d$) of constant real rank $r$; observe that $r$ must be a typical, but there exist binary forms of rank $r$ that are not typical forms (see \ref{ex:monomios}, Example 1).

For complex forms, the set of binary forms of rank exactly $r$ has non-empty interior for $r=\lfloor{\frac{d}{2}} \rfloor +1$ (see  \cite{CO}), so there is only one \textit{generic rank}. However, in the real case, all ranks between $\lfloor{ \frac{d} {2}} \rfloor +1$ and $d$ are typical (see \cite{Ble} and \cite{CO}). Hence for each typical rank $r$, the semialgebraic set $\mathcal{W}^{(r)} $ is decomposed in strata $\Gamma_j$; some of them of maximal dimension, say for $j$ in a finite set $J_r$. So the real binary forms of real rank $r$ that are stable under perturbations in any directions are described in the semialgebraic set:
\begin{equation}
\mathcal{S}_r =\bigcup_{j\in J_r } \Gamma_j  \ \subset\mathcal{W}^{(r)}  \ .
\end{equation}
In general $\mathcal{S}_r $ is not a connected set, hence a local study must be consider for each $p$ in $\mathcal{S}_r$. In general $\mathcal{S}_r$ is strictly contained in $\mathcal{W}^{(r)}$, as can be deduced from \ref{ex:monomios}. It is a forthcoming work to find local descriptions for this semialgebraic sets $\mathcal{S}_r$, in view of the multiple uses that Waring's decomposition has in both Applied Mathematics and Engineering.

\medskip

\subsection{Real projective decompositions of a given form}

\medskip

Let be  $\mathcal{B}_d$ the real vector space of real binary forms of degree $d$ in the variables $x,y$. Let be $\mathbb{P}(\mathcal{B}_d )$ the projective space over the real vector space $\mathcal{B}_d$. For $p(x,y)$  in $\mathcal{B}_d$,

\begin{equation}
p(x,y)=p_{\vec{c}}(x,y)= \sum_{i=0}^{d}                         
\left(
\begin{minipage}[c][9pt][b]{8pt}
$$  
\begin{array}{c}
   \vspace{-4pt}
   \!\! \scriptstyle d\!\! \\
   \vspace{-4pt}
    \!\! \scriptstyle i  \!\!
 \end{array} 
 $$
\end{minipage}
\right)
c_{i}\, x^{i}\, y^{d-i}, \quad \text{with } \vec{c}=(c_0,\ldots, c_d) \in \mathbb{R}^{d+1}\setminus{\vec{0}} \ ,
\end{equation}
we will associate with $p_{\vec{c}}$ the projective point $\underline{c}=[c_0:c_1:\cdots:c_d]$ in $\mathbb{P}(\mathcal{B}_d )$. Next we will fix the point $\underline{c}$, and we will study the set of all  Waring's decompositions of $p_{\vec{c}}$. For this, we will analyze the fiber $\pi^{-1} (\underline{c}) $ where $\pi$  is the projection $\pi : \mathbb{P}^{d}\times \mathbb{P}^{d} \rightarrow \mathbb{P}^{d}$ given by $\pi (\underline{x},\underline{y}) =\underline{x}  $. First we present a concrete example of the fibers $\pi^{-1} (\underline{c})\cap  \mathcal{F}^{(s)}$ for $s=\textrm{rk}_\mathbb{R}(p), \dots , d$.Then we will show the general behavior of Waring decompositions of a fixed form $p$.

\begin{ex}
Let be $p(x,y)=y^3+3 x^2 y$. Its associated pojective point in $\mathbb{P}(\mathcal{B}_3 )$ is $\underline{c}=[1:0:1:0]$. Its real rank is $2$ and we have:
$$
y^3+3 x^2 y =\frac{1}{2} (x+y)^3-\frac{1}{2} (x-y)^3 \ .
$$
It is easy to verify that $\pi^{-1} (\underline{c})\cap  \mathcal{F}^{(2)}=\left\{ ([1:0:1:0],[1:0:-1:0]) \right\}$. By direct computations we have
\begin{equation}
 \pi^{-1} (\underline{c})\cap  \mathcal{F}^{(3)}=   
\left\{
 ([1:0:1:0],[q_0 :q_1 :-q_0 :q_3 ]) \in \mathbb{P}^3 \times \mathbb{P}^3 \ : \ M_1 >0 , M_2 >0 ,
  M_3 >0
 \right\}
\end{equation}
where the $M_i$ are the following homogeneous polynomials:
$$
M_1 = q_0^2 - 2 q_1 q_3  \ ,  \ M_2 = 2 q_0^4 + q_0^2 q_1^2 - 10 q_0^2 q_1 q_3 - 4 q_1^3 q_3 - 9 q_0^2 q_3^2  
$$
and 
$$ 
M_3 = (4 q_0^4 + q_0^2 q_1^2 - 18 q_0^2 q_1 q_3 - 4 q_1^3 q_3 - 27 q_0^2 q_3^2) \cdot q_0^2 \ .
$$
Observe that $\pi^{-1} (\underline{c})\cap  \mathcal{F}^{(2)}=  
\left( \pi^{-1} (\underline{c})\cap  \mathcal{F}^{(3)}\right)  \cap 
\left(
\{ \underline{c} \} \times \left\{  \ q_1 =0 \ , \ q_3 =0 \ 
\right\}
\right)$, and also $\left(\pi^{-1} (\underline{c})\cap  \mathcal{F}^{(3)}\right) \cap \{ q_0 =0 \} =\emptyset$. In the chart $U_0 = \{ q_0 \not= 0 \} \subset \mathbb{P}^3$ we can draw the semialgebraic region corresponding to the set $\left( \pi^{-1} (\underline{c})\cap  \mathcal{F}^{(3)}\right) $, see figure \ref{fig:fibra}. 

\begin{figure}[h]
  \centering
    \includegraphics[width=0.4 \textwidth]{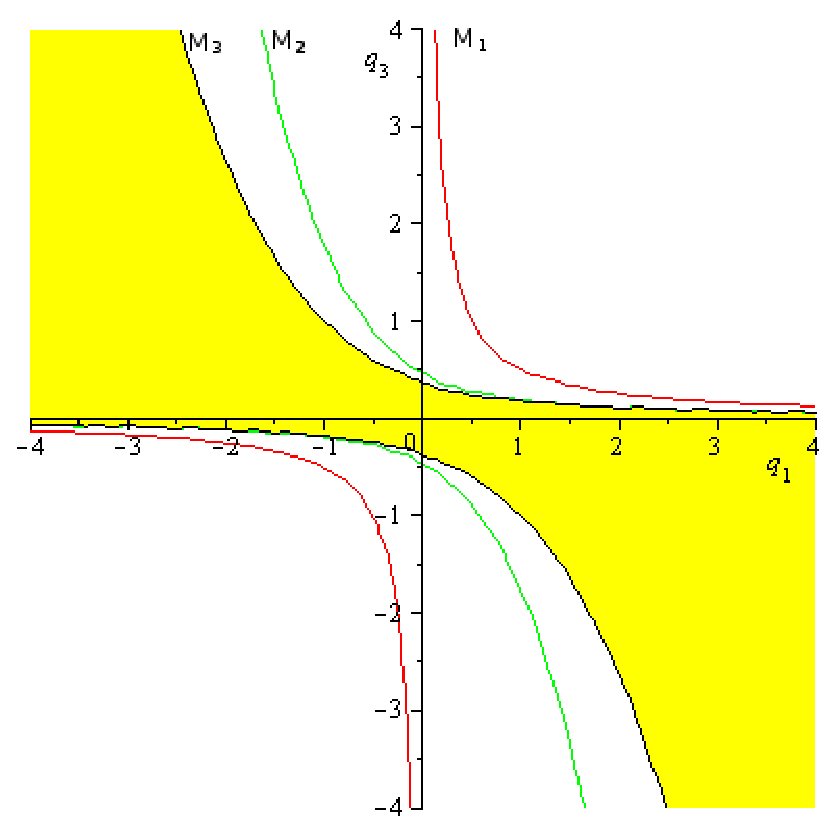}
  \caption{Shaded area corresponds with $\left( \pi^{-1} (\underline{c})\cap  \mathcal{F}^{(3)}\right) $}
  \label{fig:fibra}
\end{figure}
\medskip

We consider the semialgebraic stratification of $\pi^{-1} (\underline{c})\cap  \mathcal{F}^{(3)}$. The decomposition in strata of this set shows the discontinuity in the number of roots when moving on the one dimensional strata  towards the origin $(0,0)$ in $U_0$, where we obtain the (unique) Waring decomposition  of $p$ of length $2$.
\end{ex}

\centerline{--------------------}

\medskip

Next we present the general behavior of Waring decompositions of a fixed form $p$. Thereofer we will study  the fibers $\pi^{-1} (\underline{c})\cap  \mathcal{F}^{(s)}$ for $s=\textrm{rk}_\mathbb{R}(p), \dots , d$.

Let us consider the real projective space $\mathbb{P}^d$, and we fixed a system of projective coordinates on it; then we will write $\underline{q}=[q_0 :\dots :q_d ]$ for a point in $\mathbb{P}^d$. This space can be decompose as $\mathbb{P}^d = U_j \cup Z_{j,\infty}$, where $U_j$ is the chart of $\mathbb{P}^d$ given by $q_j\not= 0$ and $Z_{j,\infty}$ is the hyperplane $q_j = 0$. next we consider the linear subspaces given by the kernels of the Hankel matrices, $Ker(H_s)$,  associeted to $p$ as in \eqref{eq:Hr}. We embedded them in $\mathbb{P}^d$ as follow:
\begin{equation}
 \mathbb{P} (  Ker(H_s) ) \ni [q_0 :\dots :q_s ]  \rightarrow [q_0 :\dots :q_s: 0 : \dots : 0 ] \in \mathbb{P}^d \ ,
\end{equation}
for $s=\textrm{rk}_\mathbb{R}(p), \dots , d$. Obser that
\begin{equation}
     \mathbb{P} (  Ker(H_s) ) \  \subset \ \bigcap_{j>s} Z_{j,\infty} \ ,
\end{equation}
and also

\begin{equation}\label{nucleos}
     \mathbb{P} (  Ker(H_s) ) \cap U_s \  \subset \  \mathbb{P} (  Ker(H_{s+1}) )\cap Z_{s+1,\infty}
\end{equation}

From the previous formulas we have

\begin{proposition} The fibers  $\pi^{-1} (\underline{c})\cap  \mathcal{F}^{(s)}$ for $s=\textrm{rk}_\mathbb{R}(p), \dots , d$ are semianalitic sets of $\mathbb{P}^d$. Moreover, we have
\begin{equation}\label{fibras}
\left( \pi^{-1} (\underline{c})\cap  \mathcal{F}^{(s)}    \right)
\cap \left\{ \underline{c}\right\} \times U_s 
\subset 
\left( \pi^{-1} (\underline{c})\cap  \mathcal{F}^{(s+1)} \right)
\cap 
\left\{\underline{c} \right\}\times
Z_{s+1,\infty} 
\subset \pi^{-1} (\underline{c})\cap  \mathcal{F}^{(s+1)} \ .
\end{equation}

\end{proposition}

\begin{proof}
First, we observe that 
\begin{equation}
\mathbb{P} (  Ker(H_s) )=
   \left(\mathbb{P} (  Ker(H_s) ) \cap U_s \right)
   \ \cap \ 
   \left(\mathbb{P} (  Ker(H_s) ) \cap Z_{s , \infty} \right)
\end{equation}
Moreover, by \eqref{def-Fr}, we have that:
\begin{align}
    \pi^{-1} (\underline{c})\cap  \mathcal{F}^{(s)} = &(\pi^{-1} \left(\underline{c})\cap  \mathcal{A}^{(s)} \right) \cap (\underline{c} \times  \mathcal{S}^{(s)}) =\left(
    \left\{\underline{c} \right\}
    \times \mathbb{P} (  Ker(H_s) ) \right)
\cap \left(\left\{\underline{c} \right\} \times  \mathcal{S}^{(s)} \right) = \\
&\left\{\underline{c} \right\} \times  \left( \mathbb{P} (  Ker(H_s) ) \cap \mathcal{S}^{(s)} \right)
\end{align}
Hence, by \eqref{nucleos}, we have:
$$
\left( \pi^{-1} (\underline{c})\cap  \mathcal{F}^{(s)}    \right)
\cap\left( \{ \underline{c}\} \times U_s \right)
= \left\{\underline{c} \right\} \times  \left( \mathbb{P} (  Ker(H_s) ) \cap U_s \cap \mathcal{S}^{(s)} \right) 
\subset \  \left\{\underline{c} \right\} \times  \left( \mathbb{P} (  Ker(H_{s+1}) )\cap Z_{s+1,\infty} \cap \mathcal{S}^{(s)} \right) 
$$
Then, we obtain 
\begin{align*}
    \left( \pi^{-1} (\underline{c})\cap  \mathcal{F}^{(s)}    \right)
\cap\left( \{ \underline{c}\} \times U_s \right) 
\subset
&
\left\{\underline{c} \right\} \times  \left( \mathbb{P} (  Ker(H_{s+1}) )\cap Z_{s+1,\infty}  \right) 
\cap 
\left\{\underline{c} \right\} \times  \left( \mathbb{P} (  Ker(H_{s+1}) )\cap Z_{s+1,\infty} \cap \mathcal{S}^{(s+1)} \right) \subset \\
&\pi^{-1} (\underline{c})\cap  
\left(
\mathcal{A}^{(s+1)}  \cap \left\{ \underline{c} \right\} \times  \mathcal{S}^{(s+1)} 
\right)
\cap 
\left\{\underline{c} \right\}\times
Z_{s+1,\infty} \ .
\end{align*}
Therefore, we obtain the required formula.
\end{proof}

\medskip

\section{Explicit Semialgebraic decompositions}\label{sect:ex}

\subsection{Semialgebraic decomposition of  $\mathcal{B}_4$}

 Let be  $\mathcal{B}_4$ the real vector space of real binary forms of degree $4$ in the variables $x,y$. Let be $\mathbb{P}(\mathcal{B}_4 )$ the projective space over the real vector space $\mathcal{B}_4$. Let us decompose $\mathbb{P}(\mathcal{B}_4 )$ by means of the real rank function for real binary forms.To study the semialgebraic decomposition of  $\mathcal{B}_4$ we will use the following notation. Let be $r$ a positive integer,  $i_1 <i_2$ and $j_1 <j_2$  elements of  $\{ 0, \dots ,r \}$. We denote  by $Q_{i_1 i_2 , j_1 j_2}$ the $2\times 2$ minor of the Hankel matrix $H_r$ corresponding to the choice of rows $i_1$ and $i_2$  and columns $j_1$ and $j_2$, that is
 \begin{equation}\label{Qij}
  Q_{i_1 i_2 , j_1 j_2} =c_{i_1 +j_1} c_{i_2 +j_2}  -c_{i_1 +j_2}  c_{i_2 +j_1}  
 \end{equation}
Finally, we put $Q_r = \displaystyle\sum_{i_1 <i_2 , j_1 <j_2} \left(  Q_{i_1 i_2 , j_1 j_2}\right)^2$.
 
By direct computation, we obtain that $\mathcal{W}^{(1)}$ is the projective curve given by

$$
 \mathcal{W}^{(1)} =
\left\{
[1:\alpha:\alpha^2:\alpha^3 :\alpha^4]  
\ | \ \alpha \in  \mathbb{R} \right\} 
\cup 
\left\{  ^{\ }
[0 : 0 :0 :0:1] \   \right\}.
$$
Let be $\pi : \mathbb{P}^{4}\times \mathbb{P}^{4} \rightarrow \mathbb{P}^{4}$ given by $\pi (\underline{x},\underline{y}) =\underline{x}  $.
 In order to determine $\mathcal{W}^{(2)} =\mathcal{E}^{(2)} \setminus\mathcal{W}^{(1)}$, we will analyze $\mathcal{E}^{(2)} = \pi \left(\mathcal{F}^{(2)}\right)    $ with  
$$
\mathcal{F}^{(2)} = \left \{(\underline{c},\, \underline{q}) \in \mathbb{P}^4 \times \mathbb{P}^4 \  : \ 
c_{0} q_0 + c_{1} q_1 +    c_{2} q_2 =\ 0 \ 
,c_{1} q_0 + c_{2} q_1 +   c_{3} q_2 =\ 0 \ , M_B (1)>0 , M_B (2)>0  \ \right\}.
$$
where $M_B (i)$ is the $i^{th}$ principal minor of the Bezoutian matrix defided in \ref{def:bezoutian} for the polinomial $q_0 t^2 +q_1  t +q_2$. First observe that $\mathcal{F}^{(2)} \cap \{ \ q_0 =0 \ \} =\emptyset $ and $\mathcal{F}^{(2)} \cap \{ \ q_2 =0 \ \} \subset \pi^{-1} ( \mathcal{W}^{(1)} ) $. Next we proceed to find the inequalities defining $\mathcal{W}^{(2)} $. 
Then, whenever $Q_2 \not=0$ the linear system 
$$
c_{0} q_0 + c_{1} q_1 +    c_{2} q_2 =\ 0 \ 
,  \  c_{1} q_0 + c_{2} q_1 +   c_{3} q_2 =\ 0 \ , \  c_{2} q_0 + c_{3} q_1 +   c_{4} q_2 =\ 0 \ ,
$$
can be solved for $\underline{c} \not \in    \mathcal{W}^{(1)} $ and  we replace these solutions in the Bezoutian principal minors'  inequatilies
$
M_B (1)>0 , M_B (2)>0  \, .\ 
$
Next, these inequalities give the following semialgebraic description of $\mathcal{W}^{(2)} $:

\begin{equation}
\begin{array}{ll}
 \mathcal{W}^{(2)}=     &  \bigcup_{i_1 <i_2 , j_1 <j_2}
 \!\left\{
 [c_0:c_1:c_2:c_3:c_4] \in \mathbb{P}^4 \ 
  \left| \, \Delta =0 \ , \  Q_2 \not= 0 , f_{i_1 i_2 , j_1 j_2} >0\! \right.  , g_{i_1 i_2 , j_1 j_2} >0 \right\}  \cup \\
     &  \cup  \left\{ [1:\alpha:\alpha^2:\alpha^3:\beta] \in \mathbb{P}^4 \ | \ Q_2 \neq 0 \ 
  \right\}
\end{array}
\end{equation}
where $ \Delta = \det(H_2 )=c_0c_2c_4-c_0c_3^2+2c_1 c_2c_3-c_1^2c_4-c_2^3$,  $f_{i_1 i_2 , j_1 j_2}$ is obtained from $M_B (1)$ and $g_{i_1 i_2 , j_1 j_2}$ from $M_B (2)$. For instance, if $Q_{01 , 01}=c_0 c_2 -c_1^2 \neq 0$,
$$
f_{01 , 01}=
c_0 ^2 c_3 ^2-4 c_0  c_1  c_2  c_3 +2 c_0  c_2 ^3+2 c_1 ^3 c_3 -c_1 ^2 c_2 ^2\, , \quad g_{01 , 01} = f_{01 , 01}- 2Q_{01 , 01}Q_{01 , 12}
\ .
$$
To determine $\mathcal{W}^{(3)}$ we have to consider the following system of linear equations:
$$
c_{0} q_0 + c_{1} q_1 +    c_{2} q_2 +c_3 q_3 =\ 0 \, 
,  \quad  c_{1} q_0 + c_{2} q_1 +   c_{3} q_2 +c_4 q_3 =\ 0 \, , 
$$
\newpage
\noindent
that can be solved for $\underline{c} \not \in  \left( \mathcal{W}^{(1)} \cup \mathcal{W}^{(2)} \right)$ whenever $Q_3 \not=0$. Next, we must consider the Bezoutian matrix associated to $q(t)=q_0 t^3 +q_1 t^2 +q_2 t +q_3$ and its principal minors as in \eqref{eq:MB3}. In order to simplify the expressions, we take $q_3=\nu q_2$. With this trick, for instance, when $Q_{01 , 01}>0$, the condition for $p$ to be in $\mathcal{W}^{(3)}$ can be expresed by the inequalities:

$$
2\,Q_{01 , 03} \nu^2+2\,Q_{01 , 02} \nu+Q_{01 , 01}>0 \, , \ 
$$
$$
C_{14}(\underline{c})\nu^4+C_{13}(\underline{c})\nu^3+C_{12}(\underline{c})\nu^2+C_{11}(\underline{c})\nu+C_{10}(\underline{c})>0\ \textrm{ and }
$$
$$
C_{24}(\underline{c})\nu^4+C_{23}(\underline{c})\nu^3+C_{22}(\underline{c})\nu^2+C_{21}(\underline{c})\nu+C_{20}(\underline{c})>0,
$$
with
$$
\begin{array}{rl}
 C_{14}=& 9 c_ 1 ^2 c_ 2 ^2 c_ 3 ^2-4 c_ 1 ^3 c_ 3 ^3+4 c_ 0 ^3 c_ 4 ^3-12 c_ 0 ^2 c_ 1  c_ 3  c_ 4 ^2-9 c_ 0  c_ 1 ^2 c_ 2  c_ 4 ^2+12 c_ 0  c_ 1 ^2 c_ 3 ^2 c_ 4 +18 c_ 0  c_ 1  c_ 2 ^2 c_ 3  c_ 4 -9 c_ 0  c_ 2 ^3 c_ 3 ^2+9 c_ 1 ^4 c_ 4 ^2-\\
 & - 18 c_ 1 ^3 c_ 2  c_ 3  c_ 4  \\
   C_{13}=   & 18 c_ 1 ^2 c_ 2 ^3 c_ 3 -18 c_ 1 ^3 c_ 2 ^2 c_ 4 -20 c_ 1 ^3 c_ 2  c_ 3 ^2+10 c_ 0 ^2 c_ 2 ^2 c_ 3  c_ 4 +10 c_ 0  c_ 1 ^3 c_ 4 ^2+6 c_ 0  c_ 1 ^2 c_ 2  c_ 3  c_ 4 +12 c_ 0  c_ 1 ^2 c_ 3 ^3+18 c_ 0  c_ 1  c_ 2 ^3 c_ 4 +\\
   & +8 c_ 0  c_ 1  c_ 2 ^2 c_ 3 ^2-18 c_ 0  c_ 2 ^4 c_ 3 +8 c_ 1 ^4 c_ 3  c_ 4 -24 c_ 0 ^2 c_ 1  c_ 3 ^2 c_ 4 +12 c_ 0 ^3 c_ 3  c_ 4 ^2-22 c_ 0 ^2 c_ 1  c_ 2  c_ 4 ^2 \\
 C_{12}=     & -2 c_ 0  c_ 1  c_ 2 ^3 c_ 3 -9 c_ 0  c_ 2 ^5-10 c_ 1 ^4 c_ 2  c_ 4 +c_ 0 ^3 c_ 2  c_ 4 ^2+12 c_ 0 ^3 c_ 3 ^2 c_ 4 -c_ 0 ^2 c_ 1 ^2 c_ 4 ^2-46 c_ 0 ^2 c_ 1  c_ 2  c_ 3  c_ 4 -12 c_ 0 ^2 c_ 1  c_ 3 ^3+10 c_ 0 ^2 c_ 2 ^3 c_ 4 +\\
 &+ 10 c_ 0 ^2 c_ 2 ^2 c_ 3 ^2+22 c_ 0  c_ 1 ^3 c_ 3  c_ 4 +12 c_ 0  c_ 1 ^2 c_ 2 ^2 c_ 4 +16 c_ 0  c_ 1 ^2 c_ 2  c_ 3 ^2-10 c_ 1 ^3 c_ 2 ^2 c_ 3 +9 c_ 1 ^2 c_ 2 ^4-2 c_ 1 ^4 c_ 3 ^2 \\
 C_{11}=     & 10 c_ 0  c_ 1 ^2 c_ 2 ^2 c_ 3 -10 c_ 0  c_ 1  c_ 2 ^4-2 c_ 1 ^5 c_ 4 -10 c_ 1 ^4 c_ 2  c_ 3 +6 c_ 1 ^3 c_ 2 ^3+2 c_ 0 ^3 c_ 2  c_ 3  c_ 4 +4 c_ 0 ^3 c_ 3 ^3-2 c_ 0 ^2 c_ 1 ^2 c_ 3  c_ 4 -4 c_ 0 ^2 c_ 1  c_ 2 ^2 c_ 4 -\\
 &- 24 c_ 0 ^2 c_ 1  c_ 2  c_ 3 ^2+12 c_ 0 ^2 c_ 2 ^3 c_ 3 +6 c_ 0  c_ 1 ^3 c_ 2  c_ 4 +12 c_ 0  c_ 1 ^3 c_ 3 ^2 \\
   C_{10}=   &Q_{01 , 01} (c_3^2c_0^2+2c_0c_2^3-4c_0c_2c_3c_1-c_1^2c_2^2+2c_1^3c_3)\\
    C_{24}= &  C_{14}-18Q_{01 , 01}^2Q_{01 , 13} \\
    C_{23}= &  C_{13}-4Q_{01 , 01}Q_{01 , 13}2c_0c_4+7c_1c_3-9c_2^2\\
    C_{22}= & C_{12}-2Q_{01 , 01} (-4c_0c_2^2c_4-4c_0c_2c_3^2+8c_0c_1c_3c_4+9c_2^4-10c_1c_2^2c_3- 4c_1^2c_4c_2+5c_1^2c_3^2)\\
    C_{21}=  & C_{11}-2Q_{01 , 01}(-5c_0c_2^2c_3+c_0c_1c_4c_2+4c_0c_1c_3^2+4c_1c_2^3-3c_1^2c_2c_3-c_1^3 c_4)\\
    C_{20}=  & C_{10}-2Q_{01 , 01}^2Q_{01 , 12}\\
\end{array}
$$
The first inequality can be replaced by the condition:
$$
(2 c_0 ^2 c_2  c_4 -c_0 ^2 c_3 ^2-2 c_1 ^2 c_4  c_0 -c_1 ^2 c_2 ^2+2 c_1 ^3 c_3)Q_{01,03}>0
$$
and using algebraic techniques it is possible to eliminate the parameter $\nu$ in $M_B (2)$ and $M_B (3)$, but that work moves away from the size of this article.
Finally $
\mathcal{W}^{(4)}=\mathcal{B}_4 \backslash (\mathcal{W}^{(1)}\cup \mathcal{W}^{(2)}\cup \mathcal{W}^{(3)})
$; hence it is a real semialgebraic set too.

\subsection{Decomposition of canonical forms of degree 5.}\label{canonical}

P. Comon and  G. Ottaviani  in \cite{CO} have proposed two families of $5^{th}$ degree binary forms depending on two real parameters. In EACA 2016 \cite{ADZ} we presented the semialgebraic decomposition for Type I   canonical forms of degree 5.  Let us consider  Type II canonical forms, say 
\begin{equation}
    \Sigma = 
    \left\{
    p(x,y)= x(x^2 -y^2 )(x^2 +2axy +b y^2 ) \in \mathcal{B}_5 \ : \ (a,b)\in \mathbb{R}^2 
    \right\}
\end{equation}

Let be  $
\underline{c}=\left[0:\frac{-b}{5}:\frac{-a}{5}:\frac{b-1}{10}:\frac{2a}{5}:1 \right] \in \mathbb{P}^5
$ the corresponding projective point to the form $p(x,y)= x(x^2 -y^2 )(x^2 +2axy +b y^2 ) $. Observe that in the chart $U_5 = \{ c_5 \not= 0 \}$, the set $\Sigma$ is the affine plane passing by $P=(0,0,0,-\frac{1}{10},0)$ with associated  vector space generated by $v_1 =(0,0, -\frac{1}{5},0,\frac{2}{5})$ and $v_2 =(0,-\frac{1}{5},0,\frac{1}{10},0)$.

In the section \ref{descGlobal} we have shown that $ \mathbb{P}^5 = \bigcup_{r=1}^{5} \mathcal{W}^{(r)}$. Hence $\Sigma = \bigcup_{r=1}^{5} \Sigma^{(r)}$, with $ \Sigma^{(r)} = \Sigma \cap  \mathcal{W}^{(r)}$. By direct computations we obtain that $ \Sigma^{(1)} = \emptyset$ and also $\Sigma^{(2)} = \emptyset $. To study the remaining cases we define the parametrization $$\gamma (a,b) =\left(-\frac{b}{5},-\frac{a}{5},\frac{b-1}{10},\frac{2a}{5}\right) \ .$$ Therefore $\gamma( \mathbb{R}^2 )$ can be identified with $\Sigma$.

Next we describe $\Sigma^{(3)}$. By direct computations we obtain  
$\gamma^{-1}(\Sigma^{(3)}) = \left\{(a,b)\,|\, f (a,b) >0 \right\}$ where $f$ is the polynomial

$$
\begin{array}{lr}
f(a,b)=& 8192 a^{12}-19712 a^{10} b^2+7680 a^8 b^4+6560 a^6 b^6+480 a^4 b^8-\\
&-77 a^2 b^{10}+2 b^{12}-115712 a^{10} b+336640 a^8 b^3-287040 a^6 b^5+\\
&+44400 a^4 b^7-4680 a^2 b^9+142 b^{11}+78848 a^{10}+99840 a^8 b^2-\\
&-700160 a^6 b^4+700160 a^4 b^6-92940 a^2 b^8+3752 b^{10}-\\
&-287488 a^8 b+375552 a^6 b^3+311952 a^4 b^5-593208 a^2 b^7+\\
&+43192 b^9-4096 a^8+392736 a^6 b^2-673952 a^4 b^4+243410 a^2 b^6+\\
&+170652 b^8+12096 a^6 b-243056 a^4 b^3+348552 a^2 b^5-\\
&-170652 b^7+64 a^6-11840 a^4 b^2+62900 a^2 b^4-43192 b^6-\\
&-144 a^4 b+3960 a^2 b^3-3752 b^5+83 a^2 b^2-142 b^4-2 b^3,
\end{array}
$$
(see figure \ref{fig:rango3caso2}). 

\smallskip

\begin{figure}[ht]\label{fig:W3}
  \centering
    \includegraphics[width=0.6 \textwidth]{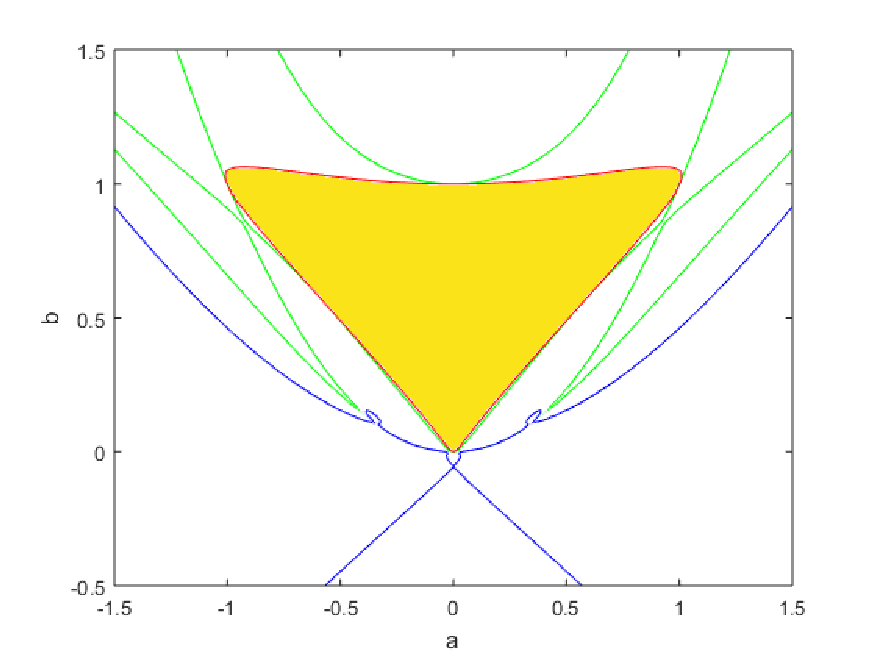}
  \caption{Shaded area corresponds with $\mathcal{W}^{(3)}$}
  \label{fig:rango3caso2}
\end{figure}

Next we consider the curve in $\sigma(t)=(a(t),b(t))=\left(1-t \ , \ 1-\frac{2}{3} \  t\right)$ for $t\in (0,1)$. Hence, we have a curve $p_t = \gamma\circ\sigma (t)$ in $\mathcal{B}_5$. By direct computations we obtain  that $p_0 (x,y)= \gamma (1,1)= x (x^2 -y^2 ) (x+1)^2$ has real rank strictly bigger that $3$, but $p_t (x,y) $ has real rank $3$ for all $t\not=0$. In fact, applying our algorithm to the family:
\begin{equation}
    p_t (x,y)=\gamma\circ\sigma (t)= x^5+(2-2t)yx^4-\frac{2 t }{3} y^2x^3+(-2+2t)y^3x^2+\left(-1+\frac{2 t }{3}\right)y^4x
\end{equation}
we get their   Waring  decompositions of length 3:

\begin{equation}
    p_t (x,y)=\frac{5t^2-9}{120t}(x+y)^5+\lambda_1\left(\frac{t-3+2\sqrt{-5t^2+6t}}{3(t-1)}x+y \right)^5+\lambda_2\left(\frac{t-3-2\sqrt{-5t^2+6t}}{3(t-1)}x+y\right)^5
\end{equation}

with 
$$
\begin{array}{lc}
  \lambda_1=-\displaystyle \frac{25t^3-30t^2+19t^2\sqrt{-5t^2+6t}-45t-60t\sqrt{-5t^2+6t}+45\sqrt{-5t^2+6t}+54}{240\, t(5t-6)}  & \text{and}  \\
   \lambda_2=- \displaystyle \frac{25t^3-30t^2-19t^2\sqrt{-5t^2+6t}-45t+60t\sqrt{-5t^2+6t}-45\sqrt{-5t^2+6t}+54}{240\, t(5t-6)}.
\end{array}
$$
We would like to point out that these decompositions do not converge to a decomposition of length 3 of the limit point $p_0$ when $t$ goes to $0$.

\newpage

Experimental computations allow us to determine some areas with rank $4$. For instance, we  have
$$
\mathcal{W}^{(4)}\supset \left\{ \ 
\left[0:\frac{-b}{5}:\frac{-a}{5}:\frac{b-1}{10}:\frac{2a}{5}:1 \right] \in \mathbb{P}^5
\ |\  F (a,b) >0  \ \right\}
$$
with $F(a,b)=-1728 a^4-1184 a^2 b^2-108 b^4+3008 a^2 b+1296 b^3-416 a^2-1224 b^2+400 b-44$.
In the Figure [\ref{fig:rango4caso2}] this set corresponds with the lined area. 

\begin{figure}[ht]
  \centering
    \includegraphics[width=0.6 \textwidth]{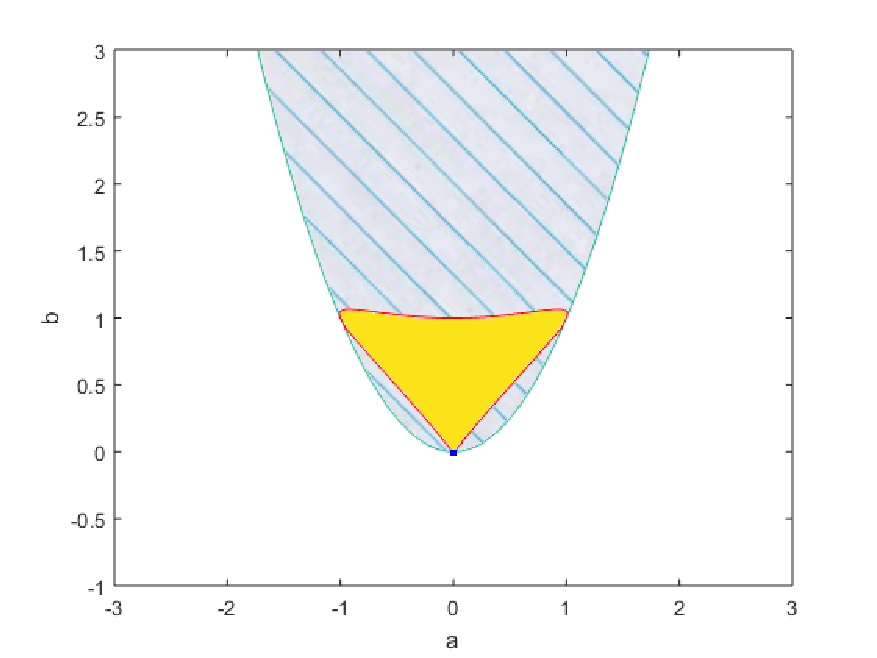}
  \caption{Lined area corresponds with $\mathcal{W}^{(4)}$}
  \label{fig:rango4caso2}
\end{figure}

\section{Appendix: An elementary proof of the Sylvester theorem}

Next, let be $\mathbf{K}$ a field of zero characteristic. We consider  the $\mathbf{K}$ vector space $\mathcal{B}_d$ for binary forms in the variables $x, y$ with coefficients in $\mathbf{K}$. For  $ q \in \mathcal{B}_d$, we denote by $q(D)$ the differential operator obtained from $q$ replacing $x$ by $\dfrac{\partial}{\partial x}$ and $y$ by $\dfrac{\partial}{\partial x}$.

\begin{lema}\label{base}
Let be $(a_1 X + b_1 Y), \ldots, (a_{d+1} X + b_{d+1} Y)$ non proportional linear forms. Then  $\{(a_1 X + b_1 Y)^d, \ldots,$ $(a_{d+1} X + b_{d+1} Y)^d \}$ is a basis of the $\mathbf{K}$ vector space $\mathcal{B}_d$.
\end{lema}  

\begin{proof} First assume  $a_i \ne 0$, for $ i =1, \ldots, d+1$. Then, it is enough to consider the case  $a_i = 1$, for $ i =1, \ldots, d+1$. Moreover, the determinant of the vectors  $\{(a_1 X + b_1 Y)^d, \ldots, (a_{d+1} X + b_{d+1} Y)^d \} \subset \mathcal{B}_d$ in the basis $\{Y^d, XY^{d-1}, \ldots, X^{d-1} Y, X^d\}$ is

\begin{equation}\label{eq:det}
\left| \begin{array}{llll}
 \binom{d}{0} b_1^d & \binom{d}{1} b_1^{d-1} & \ldots & \binom{d}{d} b_1^{d-d} \\ \ldots & \ldots & \ldots & \ldots \\
\binom{d}{0} b_i^d & \binom{d}{1} b_i^{d-1} & \ldots & \binom{d}{d} b_i^{d-d} \\ \ldots & \ldots & \ldots & \ldots \\
\binom{d}{0} b_{d+1}^d & \binom{d}{1} b_{d+1}^{d-1} & \ldots & \binom{d}{d} b_{d+1}^{d-d} \\
\end{array} \right| =
\prod_{k=0}^d \binom{d}{k}
\left| \begin{array}{llll}
b_1^d & b_1^{d-1} & \ldots & b_1^{d-d} \\ \ldots & \ldots & \ldots & \ldots \\
b_i^d & b_i^{d-1} & \ldots & b_i^{d-d} \\ \ldots & \ldots & \ldots & \ldots \\
b_{d+1}^d & b_{d+1}^{d-1} & \ldots & b_{d+1}^{d-d} \\
\end{array} \right| ,
\end{equation}
hence, it is not zero, since  $b_i$ are distinct elements in 
$\mathbf{K}$. 

On the other hand, if $a_i=0$ for some $i$, we can assume $i=1$,  $a_1=0, b_1=1$. Hence we obtain a determinant  similar to \eqref{eq:det}, but now the first row is  $(1, \ldots,0)$, and then, it is also $\not= 0$.
\end{proof}

\medskip

Let be $\displaystyle q = \sum_{k=0}^r b_k x^k y^{r-k} = \prod_{j=1}^{r} (\alpha_j x + \beta_j y) \in \mathcal{B}_d$, with $\alpha_j x + \beta_j y$ non  proportionals linear forms, and   $r \leq d+1$.

\medskip

\begin{lema}   \label{derivacion}
Let consider  $\displaystyle q =  \sum_{k=0}^r b_k x^k y^{r-k}  \in\mathcal{B}_d$ and  $\displaystyle p = \sum_{i=0}^d \binom{d}{i} a_i x^i y^{d-i}  \in \mathcal{B}_d$. Assume $r \leq d$. We have the following equivalent equations:

\begin{enumerate}
\item  $q(D) p = 0$.
\item  $\begin{pmatrix} a_0 & a_1 & \ldots & a_r  \cr  a_1 & a_2 & \ldots & a_{r+1}  \cr  \ldots & \ldots & \ddots & \ldots \cr   a_{d-r} & a_{d-r+1} & \ldots & a_d \end{pmatrix}
\begin{pmatrix} b_0 \cr b_1 \cr \vdots \cr b_r \end{pmatrix}   =
\begin{pmatrix} 0 \cr 0 \cr \vdots \cr 0 \end{pmatrix}$
\item  $\begin{pmatrix} b_0 & b_1 & \ldots & b_r & 0 & \ldots & 0 \cr  0 & b_0 & b_1 & \ldots & b_r & \ldots & 0 \cr  \ldots & \ldots & \ddots & \ddots & \ddots & \ddots & \vdots \cr  0 & \ldots & \ldots & b_0 & b_1 & \ldots  & b_r \end{pmatrix}
\begin{pmatrix} a_0 \cr a_1 \cr \vdots \cr a_d \end{pmatrix}   =
\begin{pmatrix} 0 \cr 0 \cr \vdots \cr 0 \end{pmatrix}$
\end{enumerate}
\end{lema}

\begin{proof} The statement follows from the next equalities:
\begin{align}
 q(D) p = &\left( \sum_{k=0}^r b_k \partial_x^{k} \partial_y^{r-k}  \right)   \left(  \sum_{i=0}^d \binom{d}{i} a_i x^i y^{d-i}   \right) =
 \\
  & \sum_{\tiny \begin{array}{c} k,i \\ 0 \leq i-k \leq d-r \end{array}  } \binom{d}{i} \, a_i \, b_k  \,  x^{i-k}  \, y^{d-i-r+k} \,  \frac{i!}{(i-k)!}  \,  \frac{(d-i)!}{(d-i-r+k)!}  \quad \stackrel{\tiny (m:=i-k) }{=}
  \\
  &\sum_{m=0}^{d-r}  \, \frac{d!}{ m! (d-r-m)! } \, x^m \, y^{d-r-m} \left(  \sum_{\tiny \begin{array}{c} k = 0 \\ i = k+m \end{array}  }^r a_i \, b_k   \right)  
\end{align}
Hence,    $ q(D) p = 0$ can be rewritten as the $d-r+1$ equations:

$$
a_m b_0 + a_{m+1} b_1 + \ldots + a_{m+r} b_r = 0; \; m = 0, 1, \ldots, d-r,
$$
and the lemma follows from these equations.

\end{proof}

Next, we consider the following $\mathbf{K}$ linear subspaces of $\mathcal{B}_d$:

\begin{equation}\label{subespacios}
A := \{ p \in\mathcal{B}_d \, : \, q(D) p = 0 \}, \quad 
B := \left\{ p \in\mathcal{B}_d \, : \, p = \sum_{k=1}^r \lambda_k (\beta_k x - \alpha_k y)^d \right\}    
\end{equation} 

\begin{obs}
Observe that $\dim A =r$ even when $q$ has multiple roots. Nevertheless, if this is the case, say there are  $m < r$ linearly independent linear forms, then  $\dim B = m < r$ and $B \varsubsetneq A$.
\end{obs}

\begin{lema}\label{igualdad} The linear subspaces $A$ and $B$ of $\mathcal{B}_d $ from \eqref{subespacios} are equal.
\end{lema}

\begin{proof}
It is easy to proof that $B \subseteq A$, since

$$\left( \alpha_k \frac{\partial}{\partial x} + \beta_k \frac{\partial}{\partial y} \right) (\beta_k x - \alpha_ky)^d = 0.
$$

But $B$ is generated by $\{ (\beta_k x - \alpha_k y)^d  \}_{k=1}^r$, hence  $\dim B = r$ by  \ref{base}. Moreover, by \ref{derivacion} the subspace $A$ is defined by $d-r+1$ linearly independent equations; and then $\dim A = (d+1) - (d-r+1) = r$. But $B \subseteq A$, so they equal.
\end{proof}

From lemma \ref{igualdad} we obtain the following Sylvester's theorem (The Fundamental Apolarity theorem, see \cite{Re2}). 

\begin{theorem}\label{Sylvester}
Let be $\displaystyle q = \sum_{k=0}^r b_k x^k y^{r-k} = \prod_{j=1}^{r} (\alpha_j x + \beta_j y)$ in $\mathcal{B}_d$, with some non proportional linear forms $\alpha_j x + \beta_j y$ and  $r \leq d$. Let consider $\displaystyle p = \sum_{i=0}^d \binom{d}{i} a_i x^i y^{d-i}  \in \mathcal{B}_d$.We have the following equivalent statements:
\begin{enumerate}
    \item We have the equality $q(D) p = 0$.
    \item We can rewite the binary form $p$ as:
    \begin{equation}\label{eq:waring}
     p(x) = \sum_{j=1}^r \lambda_j (\beta_j x - \alpha_j y)^d   
    \end{equation}
    for some $\lambda_j \in \mathbf{K}$.
\end{enumerate}
\end{theorem}  

\begin{obs}
This theorem \ref{Sylvester} and  the lemma \ref{derivacion} give the Sylvester's algorithm: For a binary form $p$ we try to find a binary form  without multiple roots in $\mathbf{K}$, say  $q$, of degree $ r = 1, \ldots, d$ such that $q(D) p = 0$. Then, by \ref{Sylvester}, we can obtain a Waring decomposition as in \ref{eq:waring} of minimal length.
\end{obs}

\medskip

\noindent{\bf Acknowledgements}

We wish thank for Prof. L. Gonz\'alez Vega for enlightening discussions about the decomposition of general forms of fifth degree.

\medskip


\end{document}